\documentclass[a4paper,10pt,reqno]{amsart}

\usepackage{amsthm,amsmath,amssymb}
\usepackage{amsrefs}	
\usepackage{amsfonts}
\usepackage[T1]{fontenc}
\usepackage[all]{xy}
\usepackage{enumitem}
\usepackage{caption}
\usepackage{graphicx}
\usepackage[breaklinks]{hyperref} 
\usepackage{comment}
\usepackage{pgf,tikz}
\usepackage{ifthen}
\usepackage{color}
\usetikzlibrary{arrows.meta,decorations.pathmorphing,decorations.pathreplacing,calc,intersections}
\usepackage{pgfplots}
\usepgfplotslibrary{fillbetween}
\tikzset{every picture/.style={line width=0.55pt}}%
\tikzset{>={Classical TikZ Rightarrow[scale=1.66]}}
\tikzset{
  show curve controls/.style={
	postaction={
	  decoration={
		show path construction,
		curveto code={
		  \draw [blue,opacity=0.85, line width=0.2pt] 
			(\tikzinputsegmentfirst) -- (\tikzinputsegmentsupporta)
			(\tikzinputsegmentlast) -- (\tikzinputsegmentsupportb);
		  \fill [red, opacity=0.45] 
			(\tikzinputsegmentsupporta) circle [radius=1pt]
			(\tikzinputsegmentsupportb) circle [radius=1pt];
		}
	  },
	  decorate
}}}
\definecolor{linkcolor}{rgb}{0,0,0.675}%
\hypersetup{colorlinks=true,linkcolor=linkcolor,citecolor=linkcolor,urlcolor=linkcolor,bookmarksdepth=2}%
\newtheorem{theorem}{Theorem}[section]
\newtheorem{lemma}[theorem]{Lemma}
\newtheorem{corollary}[theorem]{Corollary}
\newtheorem{proposition}[theorem]{Proposition}
\theoremstyle{definition}
\newtheorem{remark}[theorem]{Remark}
\newtheorem{definition}[theorem]{Definition}

\setenumerate{topsep=0pt plus 3pt,itemsep=0pt,leftmargin=25pt,listparindent=\parindent}
\setitemize{topsep=0pt plus 3pt,itemsep=0pt,leftmargin=25pt,listparindent=\parindent}
\setenumerate[1]{label=\normalfont(\alph{enumi})}
\setenumerate[2]{label=\normalfont(\roman{enumii}),leftmargin=10pt}
\setdescription{itemsep=0pt,leftmargin=25pt,listparindent=\parindent}

\makeatletter\@addtoreset{equation}{section}\makeatother

\renewcommand*{\P}{\mathbb P}
\newcommand*{\Spec}{\operatorname{Spec}}
\newcommand*{\image}{\operatorname{image}}
\newcommand*{\Hom}{\operatorname{Hom}}

\newcommand*{\Ext}{\operatorname{Ext}}

\newcommand*{\varTor}{\mathit{\mathcal T\hskip-2.5pt{}or}}
\newcommand*{\Div}{\operatorname{Div}}
\newcommand*{\Cl}{\operatorname{Cl}}
\newcommand*{\Pic}{\operatorname{Pic}}
\renewcommand*{\div}{\operatorname{Div}}
\newcommand*{\D}{\operatorname{D}}
\newcommand*{\Lotimes}{\mathbin{\stackrel{\mathbf{L}}{\otimes}}}
\newcommand*{\xto}[1]{\xrightarrow{\,#1\,}}
\newcommand\xmapsto[2][]{\mathrel{\mapstochar\xrightarrow[#1]{#2}}}
\makeatletter
\newcommand*{\da@rightarrow}{\mathchar"0\hexnumber@\symAMSa 4B }
\newcommand*{\da@leftarrow}{\mathchar"0\hexnumber@\symAMSa 4C }
\newcommand*{\xdashrightarrow}[2][]{%
  \mathrel{%
	\mathpalette{\da@xarrow{#1}{#2}{}\da@rightarrow{\,}{}}{}%
  }%
}
\newcommand{\xdashleftarrow}[2][]{%
  \mathrel{%
	\mathpalette{\da@xarrow{#1}{#2}\da@leftarrow{}{}{\,}}{}%
  }%
}
\newcommand*{\da@xarrow}[7]{%
  \sbox0{$\ifx#7\scriptstyle\scriptscriptstyle\else\scriptstyle\fi#5#1#6\m@th$}%
  \sbox2{$\ifx#7\scriptstyle\scriptscriptstyle\else\scriptstyle\fi#5#2#6\m@th$}%
  \sbox4{$#7\dabar@\m@th$}%
  \dimen@=\wd0 %
  \ifdim\wd2 >\dimen@
	\dimen@=\wd2 %
  \fi
  \count@=2 %
  \def\da@bars{\dabar@\dabar@}%
  \@whiledim\count@\wd4<\dimen@\do{%
	\advance\count@\@ne
	\expandafter\def\expandafter\da@bars\expandafter{%
	  \da@bars
	  \dabar@ 
	}%
  }%
  \mathrel{#3}%
  \mathrel{%
	\mathop{\da@bars}\limits
	\ifx\\#1\\%
	\else
	  _{\copy0}%
	\fi
	\ifx\\#2\\%
	\else
	  ^{\copy2}%
	\fi
  }%
  \mathrel{#4}%
}
\makeatother

\renewcommand*{\mod}{\ \operatorname{mod\,}}

\newcommand*{\Z}{\mathbb{Z}}
\newcommand*{\R}{\mathbb{R}}
\newcommand*{\Q}{\mathbb{Q}}
\newcommand*{\C}{\mathbb{C}}
\renewcommand*{\emptyset}{\varnothing}

\renewcommand*{\D}{\operatorname{D}^{\sf b}} 
\let\op\operatorname
\title{Orthogonal exceptional collections from $\Q$-Gorenstein degeneration of surfaces}
\author{Cho, Yonghwa}

\address{Korea Institute for Advanced Study, 85 Hoegiro Dongdaemun-gu, Seoul 02455, Republic of Korea}

\email{yhcho88@kias.re.kr}

\begin{document}
	\setlength{\intextsep}{0.5\baselineskip plus 2pt minus 2pt}
	\begin{abstract}
		We consider a surface that admits a $\Q$-Gorenstein degeneration to a cyclic quotient singularity $\frac{1}{dn^2}(1,dna-1)$. Under several technical assumptions, we construct $d$ exceptional vector bundles of rank $n$ which are orthogonal to each other.
	\end{abstract}

	\maketitle
	\setcounter{tocdepth}{1}
	\tableofcontents
	\section{Introduction}
		Let $\mathcal X / (0 \in \Delta)$ be a one parameter deformation of a complex normal projective surface $\mathcal X_0$ with at worst quotient singularities. The deformation $\mathcal X / (0 \in \Delta)$ is said to be $\Q$-Gorenstein if $K_\mathcal X$ is $\Q$-Cartier. Locally, a quotient singularity admitting a $\Q$-Gorenstein smoothing is either a rational double point or a cyclic quotient singularity $\frac{1}{dn^2}(1,dna-1)$ where $d, n, a>0$ are integers with $n>a>0$ and $\op{gcd}(n,a)=1$. The singularity in the latter is called a singularity of class $T$. In the case $d=1$ we call $\frac{1}{n^2}(1,na-1)$ a \emph{Wahl singularity}.
		
		On a surface with a Wahl singularity and $p_g=q=0$, $\Q$-Gorenstein smoothing of the Wahl singularity gives rise to an exceptional vector bundle\,(a vector bundle $E$ such that $\bigoplus_p \Ext^p(E,E)=\C$) on the general fiber\,(Hacking\,\cite{Hacking:ExceptionalVectorBundle}). A natural question follows:
		\begin{quote}
			What can be said if one considers a $\Q$-Gorenstein smoothing of $\frac{1}{dn^2}(1,dna-1)$ with $d>1$?
		\end{quote}
		The article concerns an answer to this question.
		\begin{theorem}[see Theorem~\ref{thm: Main thm} for details]\label{thm: Main thm Summary}
			Let $X$ be a normal projective surface with $H^1(\mathcal O_X)=H^2(\mathcal O_X)=0$, let $(P\in X) \simeq \frac{1}{dn^2}(1,dna-1)$ be a singularity of class T, and let $\mathcal X / (0 \in \Delta)$ be a $\Q$-Gorenstein smoothing of $X$. Assume there exists a Weil divisor $D \in \Cl X$ such that $D$ is Cartier except at $P$ and the image of $D$ along $\Cl X \to H_2(X;\Z) \to H_1(L;\Z)$, where $L$ is the link of $(P \in X)$, generates $H_1(L;\Z)/n^2$. After a finite base change $(0 \in \Delta') \to (0 \in \Delta)$, there exist exceptional vector bundles $E_1,\ldots,E_d$ on the general fiber $S := \mathcal X_t$ such that $\op{rank} E_k=n$ and $\Ext_S^p(E_k,E_\ell)=0$ for each $p,\ k\neq \ell$.
		\end{theorem}
		The theorem is motivated from comparing degenerations of del Pezzo surfaces and three block collections\,(\cite{KarpovNogin:3block}). Suppose $X$ is a normal projective surface with quotient singularities admitting $\Q$-Gorenstein smoothing to $\P^2$. Then, $X$ is isomorphic to either $\P^2(a^2,b^2,c^2)$, where $(a,b,c)$ is a solution to Markov equation $a^2+b^2+c^2=3abc$, or a partial smoothing of one of these weighted projective planes\,(Manetti\,\cite{Manetti:NormalDegenerationOfPlane}, Hacking-Prokhorov\,\cite{HackingProkorov:DegenerationOfDelPezzo}). It can be shown that $\P(a^2,b^2,c^2)$ has only Wahl singularities. We may apply the construction method in \cite{Hacking:ExceptionalVectorBundle} to produce three exceptional vector bundles of respective ranks $a,b,c$ on $\P^2$, which form a full exceptional collection in $\D(\P^2)$. Conversely, every full exceptional collection in $\D(\P^2)$ arises in this way\,(cf. Gorodentsev-Rudakov\,\cite{GorodenstevRudakov:ExceptionalBundleOnPlane}, Rudakov\,\cite{Rudakov:MorkovNumberAndExceptional}).

		In general, it is difficult to find equations classifying full exceptional collections in del Pezzo surfaces. We narrow down our interests into \emph{three block collections}, namely, the full exceptional collections of the form
		\[
			\langle\, E^1_1,\,\ldots,\,E^1_{d_1},\ \ E^2_1,\,\ldots,\, E^2_{d_2},\ \ E^3_1,\,\ldots,\,E^3_{d_3} \, \rangle.
		\]
		The term ``block'' refers to each subcollection $\langle\, E_1^k,\ldots,E^k_{d_k}\, \rangle$. In each block, the objects are vector bundles of the same rank, and $\Ext^p(E^k_i, E^k_j)=0$ for each $p,k$ and $i\neq j$. In three block collections, one may associate Markov type equations as follows. Let $r_k$ be the rank of the object $E^k_i$. Then, the equation
		\begin{equation}\label{eq: Markov type equations}
			d_1 r_1^2 + d_2 r_2^2 + d_3 r_3^2 = \lambda r_1r_2r_3,\quad \lambda = \sqrt{K_X^2 d_1d_2d_3}
		\end{equation}
		holds. These equations play a central role in the classification problem of three block collections(Karpov-Nogin\,\cite{KarpovNogin:3block}). The equations (\ref{eq: Markov type equations}) also emerge in the classification problem of $\Q$-Gorenstein degenerations of del Pezzo surfaces. One of the main results of \cite{HackingProkorov:DegenerationOfDelPezzo} says the following. Let $X$ be a normal projective toric surface with Picard number one that admits a $\Q$-Gorenstein smoothing to a del Pezzo surface. Then, $X$ is one of the (fake) weighted projective spaces having three singular points $\frac{1}{d_ir_i^2}(1,d_ir_ia_i-1)$. Comparing with the result of \cite{Hacking:ExceptionalVectorBundle}, it is natural to expect that each singularity contributes each block in the three block collection.
	\subsection*{Notations and conventions}{\ }\
	\begin{itemize}
		\item We always work over the field of complex numbers.
		\item For a complex analytic space $\mathcal X$ and a point $P \in \mathcal X$, $(P \in \mathcal X)$ denotes the germ of analytic neighborhoods at $P$. The readers who prefer algebraic language may replace $\mathcal X$ by an algebraic $\C$-scheme, $(P \in \mathcal X)$ by a germ of \'etale neighborhoods, and $(0 \in \Delta)$\,($\Delta$ a complex disk) by $(0 \in T)$\,($T$ an algebraic curve smooth at $0$).
		\item Let $\mu_r$ be the cyclic group generated by the $r$th root of unity $\zeta_r = \exp( 2\pi\sqrt{-1} / r)$. For an integer $a_1,\ldots,a_m$ with $\op{gcd}(r,m)=1$, we define the action
	\[
		\mu_r \times \C^m \to \C^m,\quad \zeta_r \cdot (z_1,\,\ldots,\, z_m) = (\zeta_r^{a_1} z_1,\,\ldots, \zeta_r^{a_m} z_m ).
	\]
	The singularity $(0 \in \C^m / \mu_r)$ is denoted by $\bigl(0 \in \C^m / \frac{1}{r}(a_1,\ldots,a_m)\bigr)$, or more simply, $\frac{1}{r}(a_1,\ldots,a_m)$.
		\item The Weil divisor class group is denoted by $\Cl X$, and $\mathcal O_X(D)$ is the reflexive sheaf of rank one associated to $D \in \Cl X$.
		\item Given a family $\mathcal X / (0 \in \Delta)$ and a sheaf $\mathcal E$ over $\mathcal X$, the restriction $\mathcal E\big\vert_{\mathcal X_t}$ is simply denoted by $\mathcal E_t$.
		\item For a smooth projective variety $X$, $\D(X) = \D(\mathop{\bf Coh} X)$ is the bounded derived category of coherent sheaves on $X$. Given a set of objects $A \subset \D(X)$, $\langle A \rangle$ denotes the smallest $\C$-linear full triangulated subcategory containing $A$.
	\end{itemize}
	\section{Exceptional collections}\label{sec:Exceptional collection}
	Let $V$ be a smooth projective variety, and let $\D(V)$ be the (bounded) derived category of coherent sheaves on $V$. We are interested in the decomposition of $\D(V)$ into smaller pieces. One natural attempts is to consider \emph{orthogonal decompositions}, namely, the decomposition $\D(V) = \langle \mathcal T_1, \mathcal T_2 \rangle$ such that 
	\[
		\Hom_{\D(V)}(T_1,T_2) = \Hom_{\D(V)}(T_2,T_1) = 0,\quad \text{for all}\ T_1 \in \mathcal T_1\ \text{and}\ T_2 \in \mathcal T_2.
	\]
	However, there is no orthogonal decomposition unless $V$ is disconnected\,(\cite[Proposition~3.10]{Huybrechts:FourierMukai}). Discarding the condition $\Hom_{\D(V)}(T_1,T_2)=0$, one obtains the notion of \emph{semiorthogonal decomposition}, which is more interesting than the orthogonal decompositions.
	\begin{definition}
		The full triangulated subcategories $\mathcal T_1,\mathcal T_2 \subset \D(V)$ form a \emph{semiorthogonal decomposition} if the following holds:
		\begin{enumerate}
			\item $\D(V) = \langle \, \mathcal T_1,\, \mathcal T_2 \, \rangle$;
			\item $\Hom_{\D(V)} ( T_2, T_1) = 0$ for any $T_1 \in \mathcal T_1$, $T_2 \in \mathcal T_2$.
		\end{enumerate}
		The category $\mathcal T_2$ is said to be \emph{left orthogonal} to $\mathcal T_1$, and is denoted by ${}^\perp \mathcal T_1$. Similarly, $\mathcal T_1 = \mathcal T_2^\perp$ is right orthogonal to $\mathcal T_2$.
	\end{definition}
	Once the notion of decomposition is established, one needs to find the natural candidates of the components. The simplest possible component would be the derived category $\D(\Spec \C)$ of a point. The structure of $\D(\Spec \C)$ is rather simple; indeed, every object is isomorphic to a complex of the form
	\[
		\bigoplus_p^{\text{finite}} \C^{r_p}[-p] = \bigl( \ldots \to \C^{r_{p-1}} \xto{0} \C^{r_p} \xto{0} \C^{r_{p+1}} \to \ldots \bigr).
	\]
	Assume that $\D(V)$ contains $\D(\Spec \C)$ as a semiorthogonal component. Then, the embedding $\Phi \colon \D(\Spec \C) \hookrightarrow \D(V)$ is determined by the image of the complex $\underline \C = (\ldots \to 0 \to \C \to 0 \to \ldots)$ supported at degree $0$. The object $E := \Phi(\underline \C)$ is called an \emph{exceptional object}. We may reformulate the definition intrinsically as follows.
	\begin{definition}
		An object $E \in \D(V)$ is \emph{exceptional} if
		\[
			\Hom_{\D(V)}(E,E[p]) = \left\{
				\begin{array}{ll}
					\C & p=0 \\
					0 & p\neq 0
				\end{array}
			\right.{}.
		\]
	\end{definition}
	Given an exceptional object $E_1$, one may define the semiorthogonal decompositions $\D(V) = \langle E_1, \mathcal A \rangle$, where $\mathcal A = \langle \{ A \in \mathcal \D(V) : \Hom(A, E_1[p])=0,\ \forall p \} \rangle$ is the category left orthogonal to $\langle E_1 \rangle$. In some cases, it is possible to find another semiorthogonal complement $\D(\Spec \C) \xto\sim \langle E_2 \rangle \subset \mathcal A$. Repeating this procedure, we have an ordered set of exceptional objects $E_1,\ldots,E_n \in \D(V)$ satisfying the semiorthogonality condition. This suggests the following definition.
	\begin{definition}\label{def: exceptional and orthogonal collection}\ 
		\begin{enumerate}
			\item An ordered set of exceptional objects $E_1,\ldots,E_n \in \D(V)$ is called an \emph{exceptional collection} if
			\[
				\Hom_{\D(V)} (E_{i+k},E_i[p])=0,\ \text{for every $p$ and $k > 0$}.
			\]
			An exceptional collection $E_1,\ldots,E_n \in \D(V)$ is said to be \emph{full} if
			\[
				\D(V) = \langle E_1,\ldots,E_n \rangle.
			\]
			\item An \emph{orthogonal collection} is an ordered set $(E_1,\ldots,E_n)$ of exceptional objects that is completely orthogonal, namely, 
			\[
				\Hom_{\D(V)} (E_{i+k},E_i[p])=0,\ \text{for each $p$ and $k \neq 0$}.
			\]
		\end{enumerate}
	\end{definition}
	
	\section{Singularities of class T and their $\Q$-Gorenstein deformations} 
	Let $X \subset \P^N$ be a complex surface, and let $P \in X$ be a normal singularity. We may take a sufficiently small ball $B_\P \subset \P^N $ at $P$. Let $B = B_\P \cap X$. The boundary $L$ of $B$ is called the \emph{link} of the singularity $(P \in X)$. The link is known to capture various topological nature of $(P \in X)$
	
	\begin{theorem}[{Mumford\,\cite{Mumford: TopologyOfNormalSurfaceSingularity}}]\label{thm: Mumford TopologyNSS}
		Let $\pi \colon Y \to (P \in X)$ be a resolution of $(P \in X)$. Assume that the exceptional divisor $E_1\cup \ldots\cup E_r$ of $\pi$ is simple normal crossing. Let $\alpha_i$ be a loop around $E_i$ oriented by $\displaystyle\int_{\alpha_i} f_i^{-1} df_i = 2\pi \sqrt{-1}$, where $(f_i=0)$ is a local defining equation of $E_i$. Then, the $H_1(L;\Z)$ has the group presentation:
		\[
			H_1(L;\Z) = \bigl\langle \alpha_1,\ldots,\alpha_r : \sum_{j=1}^r (E_i.E_j) \alpha_j = 0,\ \forall i=1,\ldots,r\bigr\rangle.
		\]
	\end{theorem}
	The Mayer-Vietoris sequence associated to $X = (X \setminus B) \cup B$ defines the map $H_2(X;\Z) \to H_1(L;\Z)$. Composition with the cycle class map $\op{Cl}X \to H_2(X;\Z)$ defines
	\begin{equation}\label{eq: Div Cycle map}
		\gamma \colon \op{Cl}X \to H_1(L;\Z),\quad D \mapsto \sum_j ( D' . E_j) \alpha_j.
	\end{equation}
	where $D' \in \Cl Y$ is a divisor such that $\pi_* D' = D$. Thanks to Theorem~\ref{thm: Mumford TopologyNSS}, the right hand side of (\ref{eq: Div Cycle map}) does not depend on the choice of $D'$. For instance, if one replaces $D'$ by $D'+E_1$, then
	\begin{align*}
		\sum_j ( D' + E_1 \mathbin. E_j) \alpha_j &= \sum_j ( D' . E_j) \alpha_j + \sum_j (E_1.E_j) \alpha_j \\
		&= \sum_j (D' . E_j)\alpha_j.
	\end{align*}
	The above observation leads to the following proposition.
	\begin{proposition}\label{prop: divisors having same description in the link}
		Let $D' \in \Cl Y$ and $D := \pi_* D' \in \Cl X$. There is a one-to-one correspondence
		\[
			\bigl\{ D' + \sum_j a_j E_j \in \Cl Y \mathrel: a_1,\ldots,a_r \in \Z \bigr\} \leftrightarrow \bigl\{ (k_1,\ldots,k_r) \in \Z^r : \gamma(D) = \sum_j k_j \alpha_j \bigr\}
		\]
		given by $D'' \mapsto \bigl( (D'' . E_1),\, \ldots, \, (D'' . E_r) \bigr)$.
	\end{proposition}
	\begin{proof}
		Let $D'' = D' + \sum_j a_j E_j$ and define the column vectors $\mathbf d''$, $\mathbf d'$, $\mathbf k$, $\mathbf a$ having entries as $(D''.E_i)$, $(D'.E_i)$, $k_i$, $a_i$, respectively. Let $A = \bigl( (E_i.E_j) \bigr)_{i,j}$ be the intersection matrix. Then, we have
		\[
			\mathbf{d''} = \mathbf{d'} + A \mathbf{a}.
		\]
		Now, bijectivity follows since $A$ is negative definite\,(see \cite[Corollary~2.7]{Badescu:Surfaces}).
	\end{proof}
	
	For a cyclic quotient singularity $\frac1m(1,q)$, the exceptional locus of the minimal resolution forms a chain of smooth rational curves. Let
	\begin{center}\vskip+0.33\baselineskip
		\begin{tikzpicture}
			\pgfmathsetmacro{\xunit}{2}
			\foreach \n/\s in {1/1,2/2,3/3,4/r}{
				\coordinate (A\n) at (\xunit*\n,0);
				\ifthenelse{\n=3}{\draw node at (A\n) {$\cdots$};}{
					\draw[fill=black] (A\n) circle (2pt);
					\draw node[anchor=north, yshift=-6pt] at (A\n) {$E_\s$};
				}
				\pgfmathtruncatemacro{\m}{\n-1}
				\ifthenelse{\n>1}{
					\draw[-] ([xshift=10pt] A\m) -- ([xshift=-10pt] A\n);
				}{}
			}
		\end{tikzpicture}
	\end{center}
	be the dual intersection graph of the exceptional curves, and let $b_i = -(E_i.E_i) > 0$. The numbers $m,q$ are recovered from the Hirzebruch-Jung continued fraction
	\[
		\frac{m}{q} = [b_1,\ldots,b_r] := b_1 - \frac{1}{b_2 - \frac{1}{\ldots - \frac{1}{b_r}		}}.
	\]
	\begin{proposition}\label{prop: Link algebra}
		Let $A(b_1,\ldots,b_r)$ be the matrix
		\[
			\left[
				\begin{array}{cccccc}
					b_1 & -1 & 0  & \ldots & 0 & 0 \\
					-1 & b_2 & -1 &\ldots & 0 & 0 \\
					\multicolumn{3}{c}{\vdots} & \ddots & \multicolumn{2}{c}{\vdots} \\
					0 & 0 & 0 &\ldots & b_{r-1} & -1 \\
					0 & 0 & 0  &\ldots & -1 & b_r \\
				\end{array}
			\right],
		\]
		and let $G(b_1,\ldots,b_r)$ be the group generated by the symbols $\alpha_1,\ldots,\alpha_r$, subject to the relations
		\[
			A(b_1,\ldots,b_r) \left[
				\begin{array}{c}
					\alpha_1 \\ \vdots \\ \alpha_r
				\end{array}
			\right]=0.
		\]
		Then, $\alpha_k = \det A(b_1,\ldots,b_{k-1}) \cdot \alpha_1$ for each $k=2,\ldots,r$. Furthermore,
		\[
			m = \det A(b_1,\ldots,b_r),\quad q = \det A(b_2,\ldots,b_r).
		\]
	\end{proposition}
	\begin{proof}
		We have a recurrence formula
		\[
			A(b_1,\ldots,b_k) = b_k A(b_1,\ldots,b_{k-1}) - A(b_1,\ldots,b_{k-2}),\quad k \geq 2,
		\]
		when we set $\det A(\emptyset) = 1$ conventionally. Let $n_k \in \Z_{>0}$ be the smallest positive integer that fits into the formula $\alpha_k = n_k \cdot \alpha_1$. The relations of $G(b_1,\ldots,b_r)$ reads $n_k = b_{k-1} n_{k-1} - n_{k-2}$. Since $n_1 = 1 = A(\emptyset)$ and $n_2 = b_1 = A(b_1)$, we get $n_k = A(b_1,\ldots,b_{k-1}) $. To prove the remaining identities, let $\alpha_k,\beta_k>0$ be the relatively prime integers such that
		\[
			\frac{\alpha_k}{\beta_k} = [b_k,\ldots,b_1].
		\]
		Then, $\beta_{k+1} = \alpha_k$ and $\alpha_{k+1} = b_{k+1}\alpha_k - \alpha_{k-1}$. It follows that $\alpha_k = n_{k+1} = \det A(b_1,\ldots,b_k)$. This shows $m = \det A(b_1,\ldots,b_r)$ and $q' = \det A(b_1,\ldots, b_{r-1})$, where $qq' \equiv 1 \mod m$. Applying the same argument to $[b_1,\ldots,b_k] = \gamma_k/\delta_k$, one finds $q = \det A(b_2,\ldots,b_r)$.
	\end{proof}
	Our main interest is the cyclic quotient singularities having special forms, namely, when 
	\begin{equation}\label{eq: T Sing}
		(m,q) = (dn^2,\, dna-1)\ \text{where}\ d,n,a \in \Z_{>0}\ \text{with}\ \op{gcd}(n,a)=1.
	\end{equation}
	These are the cyclic quotient singularities which admit $\Q$-Gorenstein smoothing.
	\begin{definition}\label{def: QGorDef}
		Let $X$ be a normal surface such that $K_X$ is $\Q$-Cartier, and let $\mathcal X / (0 \in \Delta)$ be a deformation of $X$. The deformation $\mathcal X / (0 \in \Delta)$ is said to be \emph{$\Q$-Gorenstein} if it is locally equivariant with respect to the deformation of an index one cover. More precisely, for each singular point $P \in X$, there exists a neighborhood $P \in \mathcal V \subset \mathcal X$ which satisfies the following: the index one cover $Z \to \mathcal V_0 := \mathcal V \cap X$ admits a deformation $\mathcal Z / (0 \in \Delta)$ to which the action $Z \to \mathcal V_0$ extends and the corresponding quotient is $\mathcal Z / (0 \in \Delta) \to \mathcal V / (0 \in \Delta)$.
	\end{definition}
	Let $u,v \in \C^2 / \frac{1}{dn^2}(1,dna-1)$ be the orbifold coordinates. After the base change $x=u^{dn}$, $y=v^{dn}$, $z = uv$, we get
	\[
		(0 \in (xy=z^{dn})) \subset \C^3_{x,y,z} \Big/ \frac{1}{n}(1,-1,a)
	\]
	The index one cover is simply obtained by replacing the ambient space by $\C^3$, hence is an $A_{dn-1}$ singularity. The versal deformation of it is
	\[
		(0 \in (xy = z^{dn} + c_{dn-2} z^{dn-2} + \ldots + c_1 z + c_0 ) \subset \C_{x,y,z}^3 \times \Delta^{dn-1}_{\underline c}.
	\]
	Thus, the equivariant deformation should be given by gathering the $z^n$ term and its powers:
	\begin{equation}\label{eq: versal Q-Gor}
		\mathcal X^{\sf ver} := (xy = z^{dn} + t_{d-1} z^{(d-1)n} + \ldots + t_1 z^n + t_0 ) \subset \C_{x,y,z}^3 \Big/ \frac1n(1,-1,a) \times \Delta^{d}_{\underline t}
	\end{equation}
	The deformation $(0 \in \mathcal X^{\sf ver}) / (0 \in \Delta^d_{\underline t})$ is called the \emph{versal $\Q$-Gorenstein deformation} of $\frac{1}{dn^2}(1,dna-1)$. Indeed, every $\Q$-Gorenstein deformation of $\frac{1}{dn^2}(1,dna-1)$ is a pullback of the versal $\Q$-Gorenstein deformation.
	
	Let $(P \in \mathcal X) / (0 \in \Delta)$ be a one parameter $\Q$-Gorenstein smoothing of $(P \in X) \simeq \frac{1}{dn^2}(1,dna-1)$. Consider a small 3-dimensional complex ball $\mathcal B \subset \mathcal X$ at $P$. Then the slice over the general fiber, namely, $\mathcal B_t := \mathcal B \cap \mathcal X_t$ for general $t \in \Delta$, is called the \emph{Milnor fiber} of the smoothing. Since the complement of the Milnor fiber is homeomorphic to $X^\circ := X \setminus \mathcal B_0$, one can construct the underlying topological space of $\mathcal X_t$ by a topological surgery illustrated in Figure~\ref{fig: Topological Surgery}.
	\begin{figure}[h!]
	\centering
	\pgfdeclarelayer{auxLayer}
	\pgfsetlayers{auxLayer,main}
	\begin{tikzpicture}[scale=1]
		\draw node at (0,3){};
		\newcommand{\PuncturedFiber}{
			\draw (90+45:40pt and 20pt) arc (90+45:360+90-45:40pt and 20pt);
			\draw[densely dotted] (0,0) ellipse [x radius=40pt, y radius=4pt];		
		}
		\PuncturedFiber
		\coordinate (CoordTempA) at (90+45:40pt and 20pt);
		\coordinate (CoordTempB) at (360+90-45:40pt and 20pt);
		\newdimen\dimTempA
		\newdimen\dimTempB
		\newdimen\dimHeight
		\pgfextractx{\dimTempA}{\pgfpointanchor{CoordTempA}{center}}
		\pgfextractx{\dimTempB}{\pgfpointanchor{CoordTempB}{center}}
		\pgfextracty{\dimHeight}{\pgfpointanchor{CoordTempB}{center}}
		\pgfmathparse{1/2*(\dimTempB-\dimTempA)}
		\pgfmathsetmacro{\xradiusLink}{\pgfmathresult}
		\newcommand{\Link}[1]{%
			\draw[red] #1 arc (180:360: \xradiusLink pt and 3pt );
			\draw[red] #1 arc (90+45:90-45: 40 pt and 10pt );
		}%
		\Link{(\dimTempA, \dimHeight)}
		\newcommand*{\CentralBall}{
			\Link{(-\xradiusLink pt,\dimHeight)}
			\coordinate (SingularP) at ($(0, 0.35) + (0,\dimHeight)$);
			\draw (-\xradiusLink pt,\dimHeight) .. controls ++(75:0.25) and ++(90+45:0.5) .. (SingularP)
			.. controls ++(90-45:0.5) and ++(180-75:0.25) .. (\xradiusLink pt, \dimHeight);
			\draw[fill=black] (SingularP) circle (1pt);
		}
		\newcommand*{\MilnorFiber}{
			\Link{(-\xradiusLink pt,\dimHeight)}
			\coordinate (SingularP) at ($(0, 0.35) + (0,\dimHeight)$);
			\draw (-\xradiusLink pt,\dimHeight) .. controls ++ (60:0.75) and ++(180-60:0.75) .. (\xradiusLink pt, \dimHeight);
			\draw[name path= Hole_1] ($(0,1.66*\dimHeight)+(135: 8pt and 2pt)$) arc (135:360+45: 8pt and 2pt);
			\draw[name path= Hole_2, opacity=0] (0,1.5*\dimHeight) ellipse [x radius = 6pt, y radius=2pt];
			\draw[name path=Hole, intersection segments={of=Hole_1 and Hole_2, sequence={R2} }];
		}
		\begin{scope}[shift={(-6,0)}]
			\PuncturedFiber
			\CentralBall
		\end{scope}
		\begin{scope}[shift={(6,0)}]
			\PuncturedFiber
			\MilnorFiber
		\end{scope}
		\begin{scope}[shift={(-2,1.5)}, rotate=45]
			\CentralBall
		\end{scope}
		\begin{scope}[shift={(2,1.5)}, rotate=-45]
			\MilnorFiber
		\end{scope}
		\draw[->] (-0.8,0.8) -- ++(180-30:1.6) node [above, midway,rotate=-30] {\tiny remove $\mathcal B_0$};
		\draw[<-] (0.8,0.8) -- ++(30:1.6) node [above, midway,rotate=30] {\tiny replace by $\mathcal B_t$};
		\draw[->, decorate, decoration={snake, amplitude=0.45pt, segment length=5pt}] (-4,0) -- (-2,0);
		\draw[->, decorate, decoration={snake, amplitude=0.45pt, segment length=5pt}] (2,0) -- (4,0);
		\draw node[anchor=north] at (-6,-1) {$ X = X^\circ \underset{L}{\cup} \mathcal B_0$};
		\draw node[anchor=north] at (0,-1) {$ X^\circ,\ L = \partial X^\circ$};
		\draw node[anchor=north] at (6,-1) {$ \mathcal X_t = X^\circ \underset{L}{\cup} \mathcal B_t$};
		\draw node[red, anchor=south] at (0,.55) {$L$};
	\end{tikzpicture}
		\caption{Topological surgery to construct the general fiber}\label{fig: Topological Surgery}
	\end{figure}
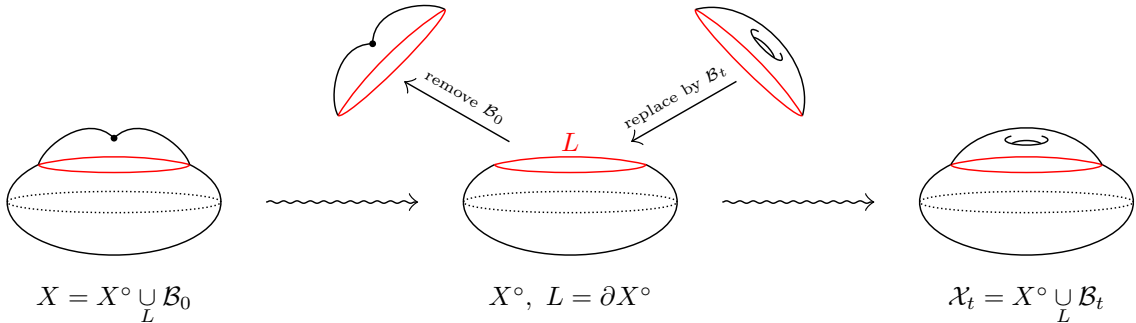	
	The ball $\mathcal B_0$ is contractible, as it is the image of a contractible ball along $\C^2 \to \C^2/\frac{1}{dn^2}(1,dna-1)$. Thus, $H_2(X,\mathcal B_0\,;\Z) \simeq H_2(X;\Z)$. Moreover, using excision principle, we see
	\[
		H_2(\mathcal X_t, \mathcal B_t\,;\Z) \simeq H_2(X^\circ, L \,; \Z) \simeq H_2(X, \mathcal B_0\,;\Z).
	\]
	Consequently, the relative homology sequence of the pair $(\mathcal X_t, \mathcal B_t)$ reads
	\begin{equation}\label{eq: Relative Homology Seq}
		\ldots \to H_2(\mathcal B_t;\Z) \to H_2(\mathcal X_t;\Z) \xto{\rm sp} H_2(X;\Z) \xto{\delta} H_1(\mathcal B_t;\Z) \to \ldots
	\end{equation}
	For a $\Q$-Gorenstein smoothing of $\frac{1}{dn^2}(1,dna-1)$, we have $H_2(\mathcal B_t) \simeq \Z^{d-1}$ and $H_1(\mathcal B_t) \simeq \Z/n\Z$\,(see for instance, \cite[Proposition~13]{Manetti:NormalDegenerationOfPlane}). A cycle $\alpha \in H_2(\mathcal X_t;\Z)$ \emph{specializes} to $X$ via the map $\mathrm{sp}$, and a cycle $\beta \in H_2(X;\Z)$ lifts to the general fiber if and only if the image under $H_2(X;\Z) \to H_1(\mathcal B_t;\Z)$ vanishes.
	\section{Wahl degeneration and exceptional bundles}
	Throughout this section, we consider the following situation: $X$ is a projective normal surface with rational singularities at worst, $H^1(\mathcal O_X)=H^2(\mathcal O_X)=0$, and $(P \in X) \simeq \frac{1}{n^2}(1,na-1)$ for $n>a>0$ with $\op{gcd}(n,a)=1$. The latter one is a particular case among singularities of class $T$, which is called a \emph{Wahl singularity}. Also, we assume that there exists a $\Q$-Gorenstein smoothing $\mathcal X / (0 \in \Delta)$. Under these assumptions, the cycle class map $c_1 \colon \Cl X \to H_2(X;\Z)$ is an isomorphism\,(cf. \cite[Proposition~4.11]{Kollar:Seifert}). The map $\delta$ in (\ref{eq: Relative Homology Seq}) can be regarded as $\Cl X \to H_1(\mathcal B_t;\Z)$. It factors through $\gamma \colon \Cl X \to H_1(L;\Z)$ in (\ref{eq: Div Cycle map}), followed by the map $H_1(L;\Z) \to H_1(\mathcal B_t;\Z)$ induced by the inclusion. Since the latter map can be identified with $\Z/n^2\Z \twoheadrightarrow \Z/n\Z$, (\ref{eq: Relative Homology Seq}) implies the following.
	\begin{equation}\label{eq: Lifting criterion}
		\text{A divisor $D_0 \in \Cl X$ lifts to $D_t \in \Cl \mathcal X_t$ if and only if $\gamma(D_0)$ is divisible by $n$.}
	\end{equation}
	Thanks to Theorem~\ref{thm: Mumford TopologyNSS}, $\gamma(D_0)$ can be computed if we know how the exceptional curves of the minimal resolution $(E \subset Y) \to (P \in X)$ intersect with the proper transform of $D_0$.
	
	However, the sequence (\ref{eq: Relative Homology Seq}) is originated from the topological surgery\,(Figure~\ref{fig: Topological Surgery}), hence a priori there is no reason for the existence of ``geometric'' deformation from $D_0$ to $D_t$. The following theorem provides a geometric interpretation via exceptional vector bundles.
	\begin{theorem}[\cite{Hacking:ExceptionalVectorBundle}]\label{thm: Hacking}
		Let $X$, $(P \in X)\simeq \frac{1}{n^2}(1,na-1)$, $\mathcal X / (0 \in \Delta)$ as above. Suppose there exists a divisor $D_0 \in \op{Cl} X$ satisfying
		\begin{enumerate}[label={\normalfont(\arabic{enumi})}]
			\item $D_0$ is Cartier except at $P$;
			\item for a resolution $Y \to (P \in X)$ with the chain $E_1 \cup \ldots \cup E_r$ of exceptional curves, $\gamma(D_0) = \alpha_1$ where $\alpha_i$ is the generator of $H_1(L;\Z)$ corresponding to $E_i$.
		\end{enumerate}
		Then, after the base change $(0 \in \Delta') \to (0 \in \Delta)$, $t \mapsto t^a$, there is a reflexive sheaf $\mathcal E$ over $\mathcal X' := \mathcal X \times_{\Delta} \Delta'$, locally free except at $P$, such that
		\begin{enumerate}
			\item $(\mathcal E_0)^{\vee\vee} = \mathcal O_X(D_0)^{\oplus n}$;
			\item for general $t \in \Delta$, $\mathcal E_t$ is an exceptional vector bundle of rank $n$;
			\item $\op{sp}(c_1(\mathcal E_t)) = c_1(nD_0)$;
			\item if $\mathcal H/(0 \in \Delta)$ is a relatively ample divisor in $\mathcal X' / ( 0 \in \Delta')$, then $\mathcal E_t$ is slope stable with respect to $\mathcal H_t'$.
		\end{enumerate}
	\end{theorem}
	This theorem is a key ingredient of this paper. To derived proper variations, we need to introduce how to construct $\mathcal E$. First of all, we may assume that $(P \in \mathcal X) / ( 0 \in \Delta)$ is a versal $\Q$-Gorenstein smoothing of $\frac{1}{n^2}(1,na-1)$, namely, it is locally isomorphic to
	\[
		(xy = z^n + t ) \subset \C^3_{x,y,z} \Big / \frac{1}{n}(1,-1,a) \times \Delta_t.
	\]
	After the base change, $\mathcal X'$ is the deformation
	\begin{equation}\label{eq: QGor Smoothing after Base change}
		(xy = z^n + t^a) \subset \C^3_{x,y,z} \Big / \frac 1n (1,-1,a) \times \Delta_t.
	\end{equation}
	Now we consider a weighted blow up as follows. First, embed the ambient space into the affine toric variety $\C^4_{x,y,z,t} / \frac 1n (1,-1,a,n)$. The corresponding fan $\Sigma$ is the orthant spanned by the rays $\R_{>0} \cdot e_1,\ldots,  \R_{>0} \cdot e_4$ in the lattice $N := \Z^4 + \Z \cdot \frac1n(1,-1,a,n)$. The weighted blow up is given by adding the ray $\R_{>0} \cdot (1,na-1,a,n)$\,(and the subsequent subdivision of cones). Since $\mathcal X' \subset \C^4 / \frac 1n (1,-1,a,n)$, we may consider the proper transform of $\mathcal X'$ along this weighted blow up. Let us denote it by $\tilde{\mathcal X}$. The exceptional divisor of $\tilde{\mathcal X}_0 \to \mathcal X_0'$ is isomorphic to
	\[
		W:= (XY = Z^n + T^a) \subset \P(1,na-1,a,n).
	\]
	Moreover, the central fiber $\tilde{\mathcal X}_0$ is the union of $\tilde X_0$\,(the proper transform of $X$), and $W$ where these components intersect scheme-theoretically along the smooth rational curve $C:=(T=0) \subset W$. On the other hand, the component $\tilde X_0$ is a partial resolution of $X$, having an irreducible exceptional locus that corresponds to $E_1$ in the resolution\,(see Figure~\ref{fig: tilde X_0}). The hypotheses on $D_0$ ensures the existence of $\tilde D_0 \in \Pic \tilde X_0$ with $(\tilde D_0 . C)=1$\,(Remark~\ref{rmk: finding desired divisor on partial resolution}).
	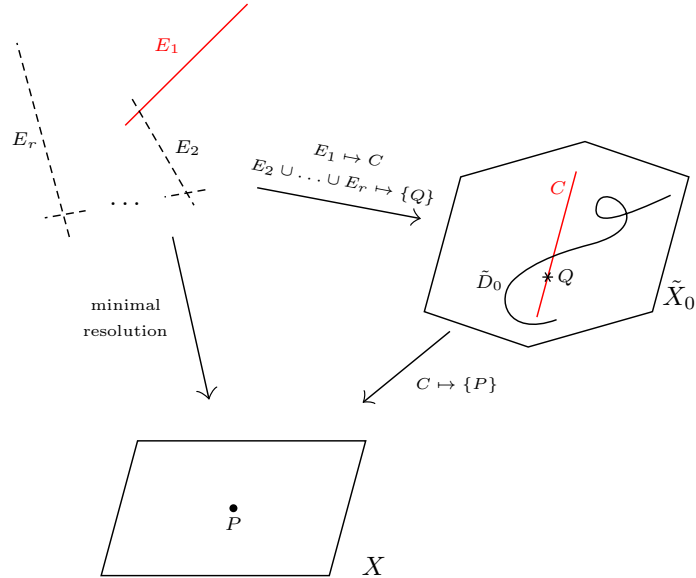
\begin{figure}[h!]
	\pgfdeclarelayer{auxLayer}
	\pgfsetlayers{auxLayer,main}
	\begin{tikzpicture}
		\draw node at (0,4){};
%
		\coordinate (posX) at (-2,-3);
		\coordinate (posPartialRes) at (2.25,0.5);
		\coordinate (posMinRes) at (-3,1.5);
		\newcommand*{\ShapeX}[2]{
			\draw[-] ($(0,0)!-0.5!($#1 + #2$)$) -- ++#1 -- ++#2 -- ++($(0,0)-#1$) -- cycle;
		}%
		\newcommand*{\SingularMark}[1]{
			\begin{scope}[shift={#1}]
				\foreach \n in {0,1,2}{
					\pgfmathtruncatemacro{\nAng}{\n*60}
					\draw[-,line width=0.5pt] (\nAng:2pt) -- (\nAng:-2pt);
				}
			\end{scope}
		}
		\newcommand*{\ShapeP}[3]{%
			\draw[-] ($(0,0)!-0.5!($#1 + #2$)$) -- ++#1 -- ++#3 --
			++($#2-#3$) -- ++($(0,0)-#1$) -- ++($(0,0)-#3$) -- cycle;
		}%
%
		\begin{scope}[shift={(posX)}]
			\draw node at (-22.5:2) {$X$};
			\ShapeX{(75:1.85)}{(3,0)}
			\fill[black] (0,0) circle (1.5pt);
			\draw node[anchor=north] at (0,0) {$\scriptstyle P$};
		\end{scope}
		\begin{scope}[shift={(posPartialRes)}]
			\draw node at (-22.5:1.75) {$\tilde X_0$};
			\ShapeP{(75:1.85)}{(3,0)}{(16:1.7)}
			\draw[-,red] (75:1) -- (75:-1);
			\draw[red] node[anchor=north east] at (75:1) {$\scriptstyle C$};
			\SingularMark{(75:-0.45)}
			\draw node[anchor=west] at (75:-0.45) {$\scriptstyle Q$};
			\draw[-]
				(0,-1) .. controls ++(180+20:0.7) and ++(270-45:0.33) .. (180+40:0.75)	
				.. controls ++(270-45:-0.4) and ++(15:-0.25) .. (0.45,0)
				.. controls ++(15:1) and ++(180+75:-0.35) .. (45:0.75)
				.. controls ++(180+75:0.35) and ++(180+25:0.66) .. (1.5,0.65);
			\draw node[anchor=east] at (180+40:0.75) {$\scriptstyle \tilde D_0$};
		\end{scope}
		\begin{scope}[shift={(posMinRes)}]
			\coordinate (Anc1) at (1,2);
			\coordinate (Anc2) at ($(Anc1) + (180+45:1.75)$);
			\coordinate (Anc3) at ($(Anc2) + (270+30:1.25)$);
			\coordinate (Anc4) at ($(Anc3) + (180+10:1.65)$);
			\coordinate (Anc5) at ($(Anc4) + (90+15:2.05)$);
			\draw[-,red] ($(Anc2)!1.15!(Anc1)$) -- ($(Anc1)!1.15!(Anc2)$);
			\foreach \n in {2,4}{
				\pgfmathtruncatemacro{\m}{\n+1}
				\draw[-,densely dashed] ($(Anc\m)!1.15!(Anc\n)$) -- ($(Anc\n)!1.15!(Anc\m)$);
			}
			\begin{scope}
				\path [clip] ($(Anc3)!0.5!(Anc4)$) circle (15pt) [insert path={(3,3) -- (3,-3) -- (-3,-3) -- (-3,3) -- (3,3)}];
				\draw[-,densely dashed] ($(Anc4)!1.15!(Anc3)$) -- ($(Anc3)!1.15!(Anc4)$);
			\end{scope}
			\draw[red] node[shift={(-0.25,0.25)}] at ($(Anc1)!0.5!(Anc2)$) {$\scriptstyle E_1$};
			\draw node[shift={(0.32,0.06)}] at ($(Anc2)!0.5!(Anc3)$) {$\scriptstyle E_2$};
			\draw node at ($(Anc3)!0.5!(Anc4)$) {$\ldots$};
			\draw node[shift={(-0.24,0)}] at ($(Anc4)!0.5!(Anc5)$) {$\scriptstyle E_r$};
		\end{scope}
		\draw [-{>[scale=2]}] ($(posMinRes)!0.2!(posX)$) -- node[anchor=east] {\tiny \begin{tabular}{c}minimal\\ resolution\end{tabular}} ($(posMinRes)!0.68!(posX)$);
		\draw [->] ($(posPartialRes)!0.33!(posX)$) -- node[anchor=north west] {\tiny $C \mapsto \{P\}$} ($(posPartialRes)!0.6!(posX)$);
		\draw [->] ($(posMinRes)!0.25!(posPartialRes)$) -- node[anchor=south,sloped]{%
			\tiny\begin{tabular}{c}
				$E_1 \mapsto C$ \\
				$E_2\cup\ldots\cup E_r \mapsto \{Q\}$
			\end{tabular}
		} ($(posMinRes)!0.66!(posPartialRes)$);
	\end{tikzpicture}
	\caption{Construction of $\tilde X_0$ in terms of $X$ and its minimal resolution}\label{fig: tilde X_0}
	\end{figure}	
	Then in the central fiber $\tilde{\mathcal X}_0 = \tilde X_0 \cup_C W$, one may construct an exceptional vector bundle by glueing $\mathcal O(\tilde D_0)^{\oplus n}$ and a certain exceptional vector bundle $G_W$\,(called ``localized exceptional bundle'' in \cite[\textsection 5]{Hacking:ExceptionalVectorBundle}) on $W$. In particular, the outcome fits into the exact sequence
	\[
		0 \to \tilde{\mathcal E}_0 \to \mathcal O_{\tilde X_0}(\tilde D_0)^{\oplus n} \oplus G_W \to \mathcal O_C(1)^{\oplus n} \to 0
	\]
	Using standard arguments in deformation theory of vector bundles, $\tilde{\mathcal E}_0$ extends to a locally free sheaf $\tilde{\mathcal E}$ on $\tilde{\mathcal X}$. The sheaf $\mathcal E$ in the statement is the reflexive hull of the pushforward along $\tilde{\mathcal X} \to \mathcal X'$.
	
	\begin{remark}\label{rmk: finding desired divisor on partial resolution}
		Let $D_Y$ be the proper transform of $D_0$ in the resolution $Y \to (P \in X)$. By Proposition~\ref{prop: divisors having same description in the link} and $\gamma(D_0) = \alpha_1$, there are integers $a_1,\ldots, a_r$ such that $D' := D_Y + a_1 E_1 + \ldots + a_r E_r$ satisfies $(D'.E_1)=1$ and $(D'.E_2)=\ldots = (D'. E_r)=0$. Now, the divisor $\tilde D_0$ is the pushforward of $D'$ along $Y \to \tilde X_0$.
		
		More generally, suppose $M_0 \in \Cl X$ is a divisor which is Cartier except at $P$ and $\gamma(M_0) = k \alpha_1$. Then, there exists $\tilde M_0 \in \Pic \tilde X_0$ such that $(\tilde M_0 . C) = k$.
	\end{remark}
	On the other hand, the same technique can be applied to construct line bundles.
	\begin{lemma}\label{lem: O_W(na-1)}
		Let $\mathcal O_W(na-1)$ be the restriction of $\mathcal O_{\P(1,na-1,a,n)}(na-1)$. Then $\mathcal O_W(na-1)$ is invertible.
	\end{lemma}
	\begin{proof}
		Let $R$ be the weighted homogeneous coordinate ring of $W$:
		\[
			R = \C[X,Y,Z,T] / (XY - Z^n - T^a),\quad \deg(X,Y,Z,T) = (1,na-1,a,n).
		\]
		Since $W$ is covered by $(X\neq0) \cup (Y \neq 0) \cup (ZT \neq 0)$, it suffices to show that for each $\mathbf{x} \in \{X,Y,ZT\}$,
		\[
			R(na-1)_{(\mathbf{x})} = \Bigl\{ \frac{f}{\mathbf{x}^d} : \deg f - d \deg \mathbf{x} = na-1 \Bigr\}
		\]
		is a free $R_{(\mathbf{x})}$-module of rank one. Indeed,
		\[
			R(na-1)_{(X)} \simeq R_{(X)} \cdot X^{na-1},\quad R(na-1)_{(Y)} \simeq R_{(Y)} \cdot Y,\quad R(na-1)_{(ZT)} \simeq R_{(ZT)} \cdot (Z^pT^q)^{na-1}
		\]
		where $p,q$ are integers with $pa+qn=1$.
	\end{proof}
	By the intersection theory of weighted projective spaces, one can prove that $\mathcal O_W(na-1)\big\vert_C = \mathcal O_C(n)$. Hence, we may glue $\mathcal O_{\tilde X_0}( n \tilde D_0)$ and $\mathcal O_W(na-1)$ along $\mathcal O_C(n)$ to obtain a line bundle $\tilde {\mathcal L}_0$ on $\tilde{\mathcal X}_0$. It extends to a line bundle $\tilde {\mathcal L}$ on $\tilde{\mathcal X}$. This construction can be generalized to the surfaces with several Wahl singularities.
	\begin{proposition}\label{prop: Hacking L.B.}
		Let $X$ be a projective normal surface with $H^1(\mathcal O_X)=H^2(\mathcal O_X)=0$. Let $M_0 \in \Cl X$ be a divisor such that
		\begin{enumerate}[label={\normalfont (\arabic{enumi})}]
			\item $M_0$ is Cartier except at Wahl singularites $(Q_i \in X) \simeq \frac{1}{n_i^2}(1,na-1)$, $i=1,\ldots,s$;
			\item $\gamma_{Q_i}(M_0) \in n_i \cdot H_1(L_i;\Z)$ where $L_i$ is the link of $(Q_i \in X)$ and $\gamma_{Q_i} \colon \Cl X \to H_1(L_i;\Z)$ is the map in (\ref{eq: Div Cycle map}).
		\end{enumerate}
		Let $\mathcal X / (0 \in \Delta)$ be a $\Q$-Gorenstein smoothing of $X$. Then, after a finite base change $(0 \in \Delta') \to (0 \in \Delta)$, there exists a reflexive sheaf $\mathcal M$ of rank $1$ over $\mathcal X' := \mathcal X \times_{\Delta} \Delta'$ such that
		\begin{enumerate}
			\item $(\mathcal M_0)^{\vee\vee} = \mathcal O_X(M_0)$;
			\item $\mathcal M_t$ is a line bundle on $\mathcal X_t'$;
			\item $\op{sp}(c_1(\mathcal M_t)) = c_1(\mathcal M_0)$.
		\end{enumerate}
	\end{proposition}
	\begin{proof}
		We make the base change $(0 \in \Delta') \to (0 \in \Delta)$ such that the ramification index is $a_i$ locally at each $(Q_i \in \mathcal X) / (0 \in \Delta)$. Let $\Phi \colon \tilde{\mathcal X} \to \mathcal X'$ be the proper birational morphism that is obtained by respective weighted blow ups at each $(Q_i \in \mathcal X')$, and let $W_1,\ldots,W_s$ be the corresponding exceptional divisors. Then, the central fiber has the irreducible decomposition $\tilde{\mathcal X}_0 = \tilde X_0 \cup W_1 \cup \ldots \cup W_s$ where $\tilde X_0$ is the proper transform of $X$. Then, by Lemma~\ref{lem: O_W(na-1)}, we have the following exact sequence of sheaves
		\[
			0 \to \tilde{\mathcal M}_0 \to \mathcal O_{\tilde X_0}(\tilde M_0) \oplus \mathcal O_{W_1}(k_1(n_1a_1-1)) \oplus \ldots \oplus \mathcal O_{W_s}(k_s(n_sa_s-1)) \to \mathcal O_{C_1}(k_1n_1) \oplus \ldots \oplus \mathcal O_{C_s}(k_sn_s) \to 0,
		\]
		where $\tilde M_0 \in \Pic \tilde X_0$ satisfies $(M_0.C_i) = k_i n_i$\, (Remark~\ref{rmk: finding desired divisor on partial resolution}), and $C_i = \tilde X_0 \cap W_i$. It can be shown that $\tilde{\mathcal M}_0$ is an exceptional line bundle, hence it deforms to $\tilde{\mathcal M} \in \Pic \tilde{\mathcal X}$ whose restriction to the general fiber is a line bundle.
	\end{proof}
	Combining Theorem~\ref{thm: Hacking} and Propositin~\ref{prop: Hacking L.B.}, we get the following corollary, which is suitable for our purpose.
	\begin{corollary}\label{cor: Hacking V.B.}
		Let $X$ be a normal projective surface with $H^1(\mathcal O_X) = H^2(\mathcal O_X)=0$. Let $P,Q_1,\ldots,Q_s \in X$ be Wahl singularities of respective indices $(n,a),\, (n_1,a_1),\,\ldots,\,(n_s,a_s)$, and let $L_*$ (resp. $\gamma_*$) be the link (resp. the map $\gamma$ in (\ref{eq: Div Cycle map})) of $* \in \{P,Q_1,\ldots,Q_s\}$. Assume there exists a divisor $D_0 \in \Cl X$ such that
		\begin{enumerate}[label={\normalfont(\arabic{enumi})}]
			\item $D_0$ is Cartier except at $\{ P,Q_1,\ldots,Q_s\}$;
			\item for a resolution $(E_1 \cup \ldots \cup E_r \subset Y) \to (P \in X)$ of the singularity, $\gamma_P(D_0) = \alpha_1$ where $\alpha_1$ is the generator corresponding to $E_1$.
			\item $\gamma_{Q_i}(D_0) \in n_i \cdot H_1(L_{Q_i};\Z)$.
		\end{enumerate}
		Let $\mathcal X / (0 \in \Delta)$ be a one parameter $\Q$-Gorenstein smoothing. Then, after a finite base change $(0 \in \Delta') \to (0 \in \Delta)$, there exists a reflexive sheaf $\mathcal E$ over $\mathcal X' = \mathcal X \times_{\Delta} \Delta'$, which is locally free over $\mathcal X ' \setminus\{P,Q_1,\ldots,Q_s\}$, satisfying the statements (a--d) in Theorem~\ref{thm: Hacking}
	\end{corollary}
	\begin{proof}
		Let $W,W_1,\ldots,W_s$ be the exceptional divisors of the weighted blow up over $P,Q_1,\ldots,Q_s$, respectively, and let $C:=\tilde X_0 \cap W$, $C_i := \tilde X_0 \cap W_i$. We may find $\tilde D_0 \in \Pic \tilde X_0$ such that $(\tilde D_0.C)=1$ and $(\tilde D_0 . C_i) = k_i n_i$. Then $\tilde{\mathcal E}_0$ can be obtained by glueing
		\[
			\mathcal O_{\tilde X_0}(\tilde D_0)^{\oplus n},\ G_W,\ \mathcal O_{W_1}(k_1(n_1a_1-1))^{\oplus n},\ \ldots,\ \mathcal O_{W_s}(k_s(n_sa_s-1))^{\oplus n}
		\]
		The remaining part is identical to the corresponding part in the proof of Theorem~\ref{thm: Hacking} or of Proposition~\ref{prop: Hacking L.B.}.
	\end{proof}
	\begin{remark}\label{rmk: comparing intersection theories}
		Assume further that $\op{Sing} X = \{P,Q_1,\ldots,Q_s\}$\,(that is, no other singularities) in the above corollary. Then, for any line bundle $\mathcal M_t$ on the general fiber, there exists $N \gg 0$ such that $\mathcal M_t^{\otimes N}$ extends to a line bundle over the whole family, i.e. there exists $\mathcal L \in \Pic \mathcal X'$ such that $\mathcal L_t = \mathcal M_t^{\otimes N}$. For instance, let $N := n^2 \prod_{i=1}^s n_i^2$, then for any $M_0 \in \Cl X$, $\gamma_P(NM_0) = 0$ and $\gamma_{Q_i}(NM_0) = 0\ \forall i$. In particular, $NM_0 \in \Pic X$. Proposition~\ref{prop: Hacking L.B.} yields a reflexive sheaf $\mathcal M$ over $\mathcal X'$ with $\mathcal M_0 = \mathcal O_X(NM_0)$. From the construction, it is obvious that $\mathcal M$ is locally free.
	   
	   One important consequence of this observation is the following: the specialization map induces the isomorphism $H_2(\mathcal X_t';\Q) \to H_2(\mathcal X;\Q)$ preserving the intersection pairing. In \textsection\ref{sec: Main}, we will see that $\Q$-Gorenstein smoothing of $\frac{1}{dn^2}(1,dna-1)$ does not have such a property.
	\end{remark}
	\section{Orthogonal collections from degenerations}\label{sec: Main}
	\subsection{Local analysis}\label{subsec: Main_ Local analysis} Let $(P \in X) \simeq \frac{1}{dn^2}(1,dna-1)$ be a singularity of class T. It corresponds to the affine toric variety that is defined by the cone $\sigma \subset N = \Z^2$ bounded by two rays $\rho_0 = \R_{>0} \cdot (0,1)$ and $\rho_d = \R_{>0} \cdot (dn^2,1-dna)$. Let $\Sigma_1$ be the fan which is obtained by subdividing the cone $\sigma$ via $\rho_1 = \R_{>0} \cdot (n^2,1-na)$. Then the cones $\sigma_1 := \op{Conv}(\rho_0,\rho_1)$ and $\sigma_2 := \op{Conv}(\rho_1,\rho_d)$ correspond to the singularities $\frac{1}{n^2}(1,na-1)$ and $\frac{1}{(d-1)n^2}(1,(d-1)na-1)$, respectively. A subsequent subdivision by $\{\rho_k = \R_{>0}\cdot(kn^2, 1-kna)\}_{k=2,\ldots,d-1}$ gives the fan $\Sigma$ whose associated toric variety has $d \times \frac{1}{n^2}(1,na-1)$.
	\begin{figure}[h!]
	\centering
	\begin{tikzpicture}
%
		\foreach \n in {0,1}{
			\ifthenelse{\n=0}{
				\coordinate (ScopeShift) at (0,0);
			}{
				\coordinate (ScopeShift) at (7,0);
			}
			\begin{scope}[shift={(ScopeShift)}]
				\coordinate (rho0 \n) at (0,1);
				\coordinate (rhod \n) at (4,-3);
				\draw[-] (0,0) -- (rho0 \n);
				\draw[-] (0,0) -- (rhod \n);
				\draw node[anchor=south] at (rho0 \n) {$\scriptstyle (0,1)$};
				\draw node[anchor=north] at (rhod \n) {$\scriptstyle (dn^2,1-dna)$};
				\draw[densely dotted] (rho0 \n) -- (rhod \n);
			\end{scope}
		}
		\foreach \k in {1}{
			\pgfmathsetmacro{\m}{1/3*\k};
			\coordinate (rho\k-1) at ($(rho0 1)!\m!(rhod 1)$);
			\draw[dashed] (ScopeShift) -- (rho\k-1);
		}
		\coordinate (Cone 0) at ($(rho0 0)!0.38!(rhod 0)$);
		\draw node[anchor=south west] at (Cone 0) {$\scriptstyle \frac{1}{dn^2}(1,dna-1)$};
		\coordinate (Cone 1-1) at ($(rho0 1)!0.165!(rhod 1)$);
		\draw node[anchor=south west] at (Cone 1-1) {$\scriptstyle \frac{1}{n^2}(1,na-1)$};
		\coordinate (Cone 1-2) at ($(rho0 1)!0.66!(rhod 1)$);
		\draw node[anchor=south west] at (Cone 1-2) {$\scriptstyle \frac{1}{(d-1)n^2}(1,(d-1)na-1)$};
		\draw node[anchor=west] at (rho1-1) {$\scriptstyle (n^2,1-na)$};
	\end{tikzpicture}
	\caption{Partial resolution that splits $\frac{1}{dn^2}(1,dna-1)$ into $\frac{1}{n^2}(1,na-1)$ and $\frac{1}{(d-1)n^2}(1,(d-1)na-1)$}\label{fig: Partial M-resolution}
	\end{figure}
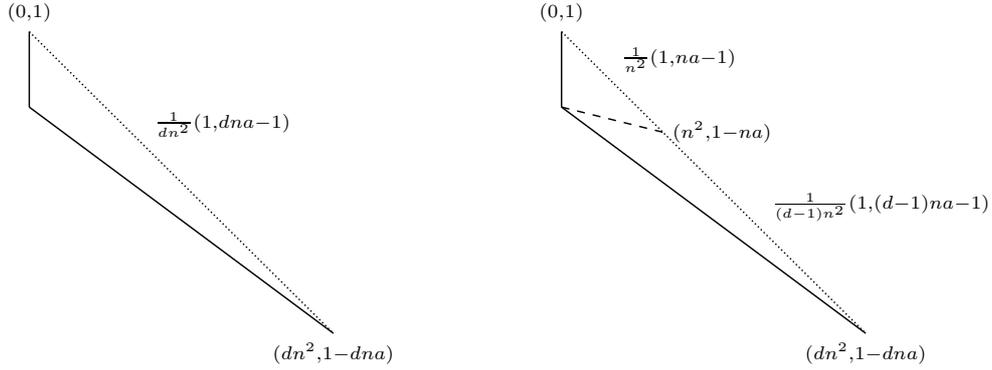
	The versal $\Q$-Gorenstein deformation $\mathcal X^{\sf ver} / (0 \in \Delta^d) $ of $(P \in X)$ is given in (\ref{eq: versal Q-Gor}). After a suitable base change $(0 \in \Delta^d) \to (0 \in \Delta^d)$, one has
	\[
		\mathcal X^{\sf ver}{}' = \bigl( xy = (z^n+t_1) ( z^{(d-1)n} + t_2 z^{(d-2)n} + \ldots + t_{d-1}z^n + t_d )\bigr) \subset \C^3 \Big/ \frac1n(1,-1,a) \times \Delta^d.
	\]
	Blowing up at the ideal $(x,z^n+t_1)$, one can deduce the following.
	\begin{proposition}\label{prop: versal simul M-res}
		Let $I = (x,z^n+t_1)$ be the ideal sheaf on $\mathcal X^{\sf ver}{}'$, and let $p \colon \mathcal Z = \op{Bl}_I \mathcal X^{\sf ver}{}' \to \mathcal X^{\sf ver}{}'$ be the blow up. Then,
		\begin{enumerate}
			\item for general $t \in \Delta^d$, $\mathcal Z_t \simeq \mathcal X_t^{\sf ver}{}'$;
			\item $p$ is an isomorphism in a neighborhood containing $\mathcal X_0^{\sf ver}{}' \setminus \{P\}$;
			\item $\mathcal Z_0$ is isomorphic to the toric variety associated to the fan $\Sigma_1$;
			\item the base space $( 0\in \Delta^d) $ admits the decomposition $(0 \in \Delta) \times (0 \in \Delta^{d-1})$ where the first (resp. second) factor parametrizes the versal $\Q$-Gorensteing deformation of $\frac{1}{n^2}(1,na-1) \in \mathcal Z_0$\ (resp. $\frac{1}{(d-1)n^2}(1,(d-1)na-1) \in \mathcal Z_0$).
		\end{enumerate}
	\end{proposition}
	This is well-known in literature. We leave a reference\,\cite[\textsection 2]{BehnkeChristophersen:MResolution}, instead of giving a full proof.
	\begin{corollary}\label{cor: simul M-res}
		Let $\mathcal X / (0 \in \Delta)$ be a one parameter $\Q$-Gorenstein smoothing of $(P \in X) \simeq \frac{1}{dn^2}(1,dna-1)$. Then after a finite base change $(0 \in \Delta') \to (0 \in \Delta)$, there exists a proper birational morphism $p \colon \bigl(A \subset \mathcal Z\bigr) \to (P \in \mathcal X')$, where $\mathcal X' := \mathcal X \times_\Delta \Delta'$, such that $(A \subset \mathcal Z_0) \to X$ is the map $\op{TV}(\Sigma) \to \op{TV}(\sigma)$ between toric varieties described in the beginning of the section \ref{subsec: Main_ Local analysis}.
	\end{corollary}
	We have $d\times \frac{1}{n^2}(1,na-1)$ in $\mathcal Z_0$ to which we apply Corollary~\ref{cor: Hacking V.B.} to construct $d$ exceptional vector bundles on the general fiber. Let $P_1,\ldots,P_d \in \mathcal Z_0$ be the singular points. They correspond to the maximal cones in the fan $\Sigma$. Let $Y \to \mathcal Z_0$ be a resolution of $P_1,\ldots,P_d$ as described in Figure~\ref{fig: min resolution of M-resolution}. The exceptional locus consists of the curves $\{E_{ij} : 1 \leq i \leq d,\ 1 \leq j \leq r\}$, where each $\{E_{ij} : 1 \leq j \leq r\}$ is the chain of rational curves contracted down to $P_i$. Also, for $k=1,\ldots,d-1$, there are proper transforms $\tilde A_k$ of the curves $A_k := \div \rho_k$ associated to $\rho_k = \R_{>0} \cdot (kn^2, 1-kna)$.
	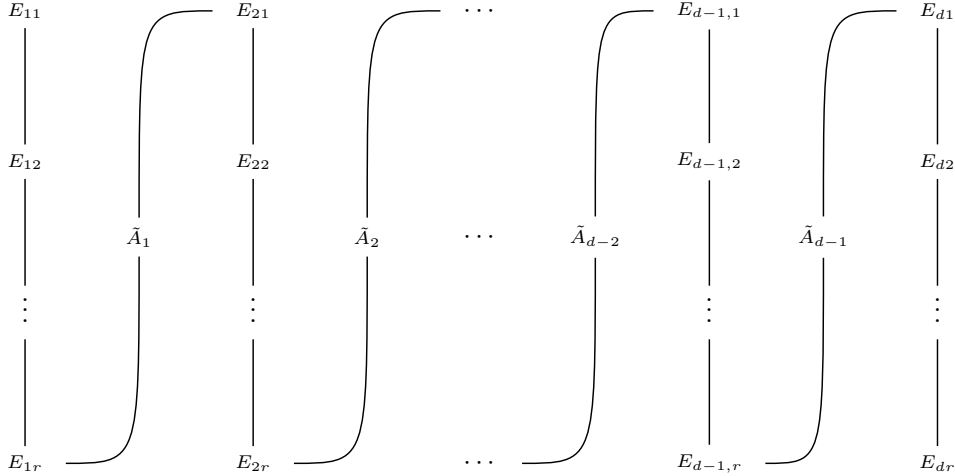
\begin{figure}[h!]
	\centering
	\begin{tikzpicture}
		\pgfmathsetmacro{\xinterval}{3}
		\pgfmathsetmacro{\yinterval}{2}
		\draw node at (0,\yinterval*1.6) {};
		\foreach \i in {1,2,3,4,5}{
			\foreach \j in {1,2,3,4}{
				\pgfmathsetmacro{\coordi}{(\i-3)*\xinterval}
				\pgfmathsetmacro{\coordj}{(2.5-\j)*\yinterval}
				\coordinate (coordE\i\j) at (\coordi,\coordj);
			}
		}
		\foreach \i/\indexi in {1/1,2/2,4/{d-1,},5/d}{
			\foreach \j in {1,2}{
				\draw node (E\i\j) at (coordE\i\j) {$\scriptstyle E_{\indexi\j}$};
			}
			\draw node (E\i3) at (coordE\i3) {\raisebox{0pt}[10pt][3pt]{$\vdots$}};
			\draw node (E\i4) at (coordE\i4) {$\scriptstyle E_{\indexi r}$};
		}
		\draw node (E31) at (coordE31) {$\cdots$};
		\draw node (E34) at (coordE34) {$\cdots$};
		\draw node at ($(coordE31)!0.5!(coordE34)$) {$\cdots$};
		\foreach \i/\indexi in {1/1,2/2,3/{d-2},4/{d-1}}{
			\pgfmathtruncatemacro{\m}{\i+1}
			\coordinate (coordA\i) at ($(coordE\i1)!0.5!(coordE\m4)$);
			\draw node (A\i) at (coordA\i) {$\scriptstyle \tilde A_{\indexi}$};
			\draw[-] ([xshift=5pt]E\i4.east) .. controls ($(E\i4)!0.5!(E\m4)$) .. (A\i.south);
			\draw[-] (A\i.north) .. controls ($(E\i1)!0.5!(E\m1)$) .. ([xshift=-5pt]E\m1.west);
		}
		\foreach \i in {1,2,4,5}{
			\foreach \j in {1,2,3}{
				\pgfmathtruncatemacro{\jj}{\j+1}
				\draw[-] (E\i\j) -- (E\i\jj);
			}
		}
	\end{tikzpicture}
	\caption{The dual intersection graph of the resolution $Y \to \mathcal Z_0$.}\label{fig: min resolution of M-resolution}
	\end{figure}
	Let $L_{P_i} \subset \mathcal Z_0$ be the link of $(P_i \in \mathcal Z_0)$. Then $H_1(L_{P_i};\Z)$ is generated by the loops $\alpha_{i1},\ldots,\alpha_{ir}$ around $E_{i1},\ldots,E_{ir}$, respectively. We consider the product of the maps in (\ref{eq: Div Cycle map}):
	\begin{equation}\label{eq: M-resolution Cycle map}
		\gamma \colon \Cl \mathcal Z_0 \to \bigoplus_{i=1}^d H_1(L_{P_i};\Z).
	\end{equation}
	\subsection{Application to the global case}
		This subsection treats the proof of the main theorem:
	\begin{theorem}\label{thm: Main thm}
		Let $X$ be a normal projective surface with $H^1(\mathcal O_X)=H^2(\mathcal O_X)=0$, and let $(P \in X)$ be a singular point $\frac{1}{dn^2}(1,dna-1)$ of class T (or $A_{d-1}$ in the case $n=a=1$). Assume that there exists a divisor $D \in \Cl X$ such that for each $Q \in \op{Sing} X$ and $\gamma_Q \colon \Cl X \to H_1(L_Q;\Z)$,
		\begin{enumerate}[label={\normalfont (\arabic{enumi})}]
			\item $\gamma_P(D)$ generates $H_1(L_P;\Z)/n^2$;
			\item $\gamma_Q(D) \in n_Q \cdot H_1(L_Q;\Z)$ if $Q$ is a Wahl singularity with indices $(n_Q,a_Q)$;
			\item $\gamma_Q(D) = 0$ otherwise.
		\end{enumerate}
		Let $\mathcal X / (0 \in \Delta)$ be a one parameter $\Q$-Gorenstein smoothing of $X$. Then, after a finite base change $(0 \in \Delta') \to (0 \in \Delta)$, there exist reflexive sheaves $\mathcal E_1,\ldots,\mathcal E_d$ over $\mathcal X' := \mathcal X \times_\Delta \Delta'$ such that
		\begin{enumerate}
			\item $\mathcal E_k$ is a reflexive sheaf of rank $n$, locally free over $\mathcal X' \setminus \op{Sing} \mathcal X_0'$;
			\item $\mathcal E_{k,0}^{\vee\vee} \simeq \mathcal O_X( m D )$ for some $m > 0$ independent of $k$;
			\item for general $t$, $\mathcal E_{k,t}$ is an exceptional vector bundle of rank $n$;
			\item $\langle\, \mathcal E_{1,t},\ldots,\mathcal E_{d,t}\, \rangle \subset \D(\mathcal X_t')$ is an orthogonal collection.
		\end{enumerate}
	\end{theorem}
	After a base change $\mathcal X' = \mathcal X \times_\Delta \Delta'$, we have a map $p \colon \mathcal Z \to \mathcal X'$ which extends $( A_1 \cup \ldots \cup A_{k-1} \subset \mathcal Z_0) \to (P \in \mathcal X')$ in Corollary~\ref{cor: simul M-res}.
	\begin{lemma}
		Assume the hypotheses of Theorem~\ref{thm: Main thm}. Let $P_1,\ldots,P_d \in \mathcal Z_0$ be the Wahl singularities $\frac{1}{n^2}(1,na-1)$ lying over $(P \in X)$, let $\gamma \colon \Cl \mathcal Z_0 \to \bigoplus_{i=1}^d H_1(L_{P_i};\Z)$ be as in (\ref{eq: M-resolution Cycle map}), and let $D' \in \Cl \mathcal Z_0$ be the proper transform of $D$. Then, there exist integers $m,a_1,\ldots,a_{d-1}$ such that $D_0 := mD'+\sum_{k} a_k A_k$ satisfies
		\[
			\gamma(D_0) = (\alpha_{11},0,0,\ldots,0),
		\]
		where $\alpha_{11}$ is the generator corresponding to $E_{11}$ in Figure~\ref{fig: min resolution of M-resolution}.
	\end{lemma}
	\begin{proof}
		Let $Y \to \mathcal Z_0$ be a resolution as in Figure~\ref{fig: min resolution of M-resolution}, and let $\beta_k \in H_1(L_P;\Z)$ be the generator corresponding to $\tilde A_{k}$. By Proposition~\ref{prop: Link algebra}, $\beta_1 = n^2 \alpha_{11}$ in $H_1(L_P;\Z)$. Thus, there exists $c \in \Z$ such that
		\[
			\gamma_P(mD) = (1 + c n^2) \alpha_{11} = \alpha_{11} + c \beta_1,
		\]
		where $\gamma_P \colon \Cl X \to H_1(L_P;\Z)$. Let $D'' \in \Cl Y$ Be the proper transform of $D$ along the composition $Y \to \mathcal Z_0 \to X$. Then, by Proposition~\ref{prop: divisors having same description in the link}, there are integers $\{a_k\}, \{e_{ij}\}$ such that
		\[
			D_0'':= mD'' + \sum_k a_k A_k + \sum_{i,j} e_{ij} E_{ij}
		\]
		satisfies $(D_0'' . E_{11})=1$ and $(D_0''.E_{ij})=0$ for $(i,j) \neq (1,1)$. The push-forward of $D_0''$ along $Y \to \mathcal Z_0$ is the desired $D_0$ in the statement.
	\end{proof}
	Let $D_k := D_0 + A_1 + \ldots + A_k$\,($k<d$). Then by Proposition~\ref{prop: Link algebra}
	\begin{align*}
		\gamma(D_i) &= \bigl( \alpha_{11} + \alpha_{1r},\, \ldots,\, \alpha_{k1} + \alpha_{kr},\, \alpha_{k+1,1}, \, 0,\,\ldots,\,0 \bigr)
		\\
		&= \bigl( na\cdot\alpha_{1r},\, \ldots,\,na\cdot\alpha_{kr},\, \alpha_{k+1,1},\, 0,\,\ldots,\,0 \bigr).
	\end{align*}
	Hence, the hypotheses of Corollary~\ref{cor: Hacking V.B.} are satisfied when one plugs $P \mapsto P_{k+1}$ and $D_0 \mapsto D_k$ into the corollary. This yields an reflexive sheaf $\mathcal F_k$ over $\mathcal Z$\,(after a base change, but we suppress it in our notation), whose restriction to the general fiber is an exceptional vector bundle $\mathcal F_{k,t}$. We claim that $\mathcal E_k = (p_* \mathcal F_k)^{\vee\vee}$ is the desired one in Theorem~\ref{thm: Main thm}. We need to verify
	\begin{equation}\label{eq: Orthogonality}
		\Ext^p(\mathcal F_{k,t},\, \mathcal F_{\ell,t})=0,\quad \forall p,\ \forall k\neq \ell,
	\end{equation}
	for general $t$.
%
	\begin{proposition}
		$\chi( \mathcal F_{k,t},\, \mathcal F_{\ell,t}) = \sum_p (-1)^p \dim \Ext^p(\mathcal F_{k,t},\, \mathcal F_{\ell,t})=0$.
	\end{proposition}
	\begin{proof}
		Since the major part of the proof involves tedious numerical computations, we give an outline instead of giving all the details. By Riemann-Roch formula, $\chi(\mathcal F_{k,t},\, \mathcal F_{\ell,t})$ is determined by Chern classes of $\mathcal F_{k,t}^\vee \otimes \mathcal F_{\ell,t}$ and their intersection products. By Remark~\ref{rmk: comparing intersection theories}, it suffices to look at their specialization. We have $\op{sp}(c_1(\mathcal F_{i,t})) = n( D_0 + A_1 + \ldots + A_{i-1})$, hence
		\[
			\op{sp}(c_1( \mathcal F_{k,t}^\vee \otimes \mathcal F_{\ell,t}) ) = \left\{
				\begin{array}{ll}
					n^2( A_k + \ldots + A_{\ell-1} ) & \text{if }\ell > k \\
					-n^2( A_\ell + \ldots + A_{k-1} ) & \text{if } \ell < k.
				\end{array}
			\right.
		\]
		We remark that the curves $A_1,\ldots,A_{d-1}$ are exceptional curves of $\mathcal Z_0 \to X$, whose intersection properties are completely determined by the local geometry at $(P \in X)$. Using toric geometry, one may verify
		\[
			(A_i)^2 = -\frac{2}{n^2},\ (A_i . A_{i+1}) = \frac{1}{n^2},\ \text{and\ } (K_{\mathcal Z_0} . A_i)=0.
		\]
		This determines $c_1^2$ and $(K_{\mathcal Z_0} . c_1)$ of $\mathcal F_{k,t}^\vee \otimes \mathcal F_{\ell,t}$; indeed, $c_1^2 = -2n^2$, $(K.c_1)=0$. Riemann-Roch formula and $\chi(\mathcal F_{k,t}, \mathcal F_{k,t})=1$ implies 
		\[
			c_2(\mathcal F_{k,t}) = \frac{n-1}{2n}(c_1^2(\mathcal F_{k,t}) + n^2 + 1). \tag{see \cite[Lemma~5.3]{Hacking:ExceptionalVectorBundle}}
		\]
		By the splitting principle, $c_2(\mathcal F_{k,t}^\vee \otimes \mathcal F_{\ell,t})$ can be expressed in terms of the Chern classes of $\mathcal F_{k,t}$ and $\mathcal F_{\ell,t}$. More specifically, for any vector bundles $\mathcal E$ and $\mathcal G$ of ranks $e,g$ respectively,
		\[
			c_2(\mathcal E \otimes \mathcal G) = \frac{g(g-1)}{2} \mathop{c_1^2}(\mathcal E) + (eg-1) \mathop{c_1}(\mathcal E) . \mathop{c_1}(\mathcal G) + \frac{e(e-1)}{2} \mathop{c_1^2}(\mathcal G) + g \mathop{c_2}(\mathcal E) + e \mathop{c_2}(\mathcal G).
		\]
		In particular, we get $c_2(\mathcal F_{k,t}^\vee \otimes \mathcal F_{\ell,t})=0$. Riemann-Roch formula reads
		\[
			\chi(\mathcal F_{k,t},\mathcal F_{\ell,t}) = n^2 \chi(\mathcal O_{\mathcal X'_t}) + \frac 12 (c_1^2 - (K.c_1)) - c_2 = 0. \qedhere
		\]
	\end{proof}
	It remains to prove (\ref{eq: Orthogonality}) for $p=0,2$. To prove the case $p=0$, we need to take stability conditions into account.
	\begin{proposition}\label{prop: Stability}
		Let $\mathcal H / (0 \in \Delta)$ be a flat family of ample divisors in $\mathcal X' / (0 \in \Delta')$. For sufficiently large $m > 0$ and $t \neq 0$, $\mathcal F_{k,t}$ is slope stable with respect to $\mathcal A_t := K_{\mathcal X'_t} + (nm)\mathcal H_t$.
	\end{proposition}
	\begin{proof}
		If $d=1$, namely, $(P \in X) \simeq \frac{1}{n^2}(1,na-1)$, then $\mathcal F_{1,t}$ is slope stable with respect to $\mathcal H_t$\,(\cite[Proposition~4.4]{Hacking:ExceptionalVectorBundle}). For $d>1$, let $\mathcal H'$ be the pullback of $\mathcal H$ along $p \colon \mathcal Z \to \mathcal X'$. Then, $\mathcal H'$ is relatively nef over $(0 \in \Delta')$. Since $\op{Nef}(\mathcal Z/ (0 \in \Delta')) = \op{\overline{Ample}}(\mathcal Z/(0 \in \Delta'))$, we may consider a sequence $\{\mathcal H'_i\}_{i \in \Z_{>0}}$ of ample $\Q$-Cartier divisors which converges to $\mathcal H'$. Using \cite[Proposition~4.4]{Hacking:ExceptionalVectorBundle}, we see that $\mathcal F_{k,t}$ is slope stable with respect to $\mathcal H_i'$ for $t\neq 0$. Hence, $\mathcal F_{k,t}$ is slope stable with respect to all $\mathcal H_i'$. For any proper quotient $\mathcal F_{k,t} \twoheadrightarrow Q$, the stability conditions says $\mu_{\mathcal H'_{i,t}}(\mathcal F_{k,t}) < \mu_{\mathcal H'_{i,t}}(\mathcal Q)$. Taking limit $i \to \infty$ in both sides, we get $\mu_{\mathcal H'_t}(\mathcal F_{k,t}) \leq \mu_{\mathcal H'_t}(\mathcal Q)$, showing that $\mathcal F_{k,t}$ is slope semistable with respect to $\mathcal H'_t = \mathcal H_t$. This shows that $\mathcal F_{k,t}$ is semistable with respect to $\mathcal H_t$ for $t \neq 0$.
		
		Let $\mathcal A = K_{\mathcal X'} + (nm) \mathcal H$ for sufficiently large $m$, so that $\mathcal A_0$ is ample in $X=\mathcal X_0'$. The previous argument can be applied to $\mathcal A$, thus $\mathcal F_{k,t}$ is slope semistable with respect to $\mathcal A_t$. To prove the stability, it suffices to prove that $(c_1(\mathcal F_{k,t}) \mathbin . \mathcal A_t)$ is relatively prime to $n$. By Remark~\ref{rmk: comparing intersection theories}, we may compute the intersection in the central fiber via the specialization map. Since $c_1(\mathcal F_{k,t}) \xmapsto{\rm sp} n( D_0 + A_1 + \ldots + A_{k-1}) $, $(\mathcal A_t \mathbin. c_1(\mathcal F_{k,t}) ) = (
		K_{\mathcal Z_0} + nm \mathcal H'_0 \mathbin. nD_0)$. Here, $(\mathcal H'_0 \mathbin. nD_0) = (\mathcal H_t \mathbin. c_1(\mathcal F_{k,t})) \in \Z$, so
		\[
			(\mathcal A_t \mathbin. c_1(\mathcal F_{k,t})) = ( K_{\mathcal Z_0} \mathbin . nD_0 )\ (\mathrm{mod}\ n)
		\]
		Let $I \in \Z/n\Z$ be the intersection number $( K_{\mathcal Z_0} \mathbin . nD_0 )$ modulo $n$. Then, $I$ purely depends on the local geometry of $(A_1\cup\ldots\cup A_{k-1} \subset \mathcal Z_0) \to (P \in X)$, hence we may identify $\mathcal Z_0$ with the toric variety associated to the fan $\Sigma$ consisting of the rays $\{ \rho_k := \R_{>0} \cdot ( kn^2, 1-kna) : k=0,1,\ldots,d \}$ and the cones $\sigma_k := \op{Cone}(\rho_{k-1},\rho_k)$. In this setup, $A_k\,(k=1,\ldots,d-1)$ are the divisors associated to the rays $\rho_k$. Let $A_0 \subset (A_1\cup \ldots \cup A_{k-1} \subset \mathcal Z_0)$ be the local curve associated to the ray $\rho_0$. Then, $\gamma(D_0) = (\alpha_{11},0,\ldots,0) = \gamma(A_0)$, thus $D_0 - A_0$ is Cartier. This shows that 
		$(D_0 - A_0 \mathbin. K_{\mathcal Z_0} ) \in \Z$. Hence, $I = (nA_0 \mathbin. K_{\mathcal Z_0} ) = -a \ (\mathrm{mod}\ n)$\,(see \cite[p.1192]{Hacking:ExceptionalVectorBundle}). Since $(\mathcal A_t \mathbin . c_1(\mathcal F_{k,t})) = I\ (\mathrm{mod}\ n)$, $(\mathcal A_t \mathbin. c_1(\mathcal F_{k,t}))$ is relatively prime to $n$.
	\end{proof}
	During the proof, we have shown that $(\mathcal A_t \mathbin. c_1(\mathcal F_{k,t}) ) = (K_{\mathcal Z_0} + nm \mathcal H'_0 \mathbin. nD_0)$ is independent of $k$. This implies:
	\begin{corollary}
		For $k \neq \ell$, $\Hom(\mathcal F_{k,t},\, \mathcal F_{\ell,t} ) =0$.
	\end{corollary}
	\begin{proposition}
		For $k \neq \ell$, $\Ext^2(\mathcal F_{k,t},\, \mathcal F_{\ell,t} ) = 0$.
	\end{proposition}
	\begin{proof}
		By Serre duality, it suffices to prove that $H^0( \mathcal F_{k,t} \otimes \mathcal F_{\ell,t}^\vee \otimes \omega_{\mathcal X'_t} ) = 0$. Let $\mathcal G$ be the reflexive hull of $p_* \bigl( \mathcal F_k \otimes \mathcal F_\ell^\vee \otimes \omega_{\mathcal Z / \Delta} \bigr)$\,(recall: $p \colon \mathcal Z \to \mathcal X'$ is the morphism defined in Corollary~\ref{cor: simul M-res}). For general $t$, $\mathcal G_t = \mathcal F_{k,t} \otimes \mathcal F_{\ell,t}^\vee \otimes \omega_{\mathcal X'_t}$\,(cf. \cite[Theorem~3.6.1]{Conrad:GrothendieckDuality}), hence our claim is $H^0(\mathcal G_t)=0$. We divide the proof into several steps.
		\begin{enumerate}[label={\it\indent Step~\arabic{enumi}.}, fullwidth]
			\item We first claim that $H^0(\mathcal G_0) =0$. Let $p_0 \colon \mathcal Z_0 \to X$ be the morphism between central fibers. By Proposition~\ref{prop: versal simul M-res}(b), $p_0$ induces the isomorphism $\mathcal Z_0 \setminus ( A_1 \cup \ldots A_{k-1}) \xto\sim X \setminus \{P\}$. Let us denote $X \setminus\{P\}$ by $X^\circ$. By the construction of $\mathcal F_k$, we have $(p_0{}_* \mathcal F_{k,0})^{\vee\vee} = \mathcal O_X( p_0{}_* D_0)^{\oplus n}$ for all $k=1,\ldots,d$. Hence,
		\begin{align*}
			\mathcal G\big\vert_{X^\circ} & \simeq \bigl( \mathcal O_X( p_0{}_* D_0 ) )^{\oplus n} \otimes \mathcal O_X( -p_0{}_* D_0 ) )^{\oplus n} \bigr)^{\vee\vee} \big\vert_{X^\circ} \otimes \omega_{X^\circ} \\
			& \simeq \omega_{X^\circ}^{\oplus n^2}.
		\end{align*}
		Since $\mathcal G$ is reflexive, it satisfies the condision $S_2$ of Serre, hence $\mathcal G_0$ satisfies the conditions $S_1$\,(torsion-freeness). Thus, $\mathcal G_0 \hookrightarrow (\mathcal G_0)^{\vee\vee} = \omega_X^{\oplus n^2}$, and $h^0(\mathcal G_0) \leq n^2\,h^0(\omega_X) = 0$.
			\item It is natural to expect that $H^0(\mathcal G_0)=0$ implies $H^0(\mathcal G_t)=0$. However, $\mathcal G$ is not flat over $\Delta'$, so it does not automatically follows from the semicontinuity. Let $\varphi \colon \mathcal X' \to \Delta'$ be the deformation morphism. Then we may define an upper-semicontinuous function
			\begin{equation}\label{eq: local rank function}
				h \colon \Delta' \to \Z_{>0},\quad t \mapsto \dim_\C \varphi_* \mathcal G \otimes \mathcal O_t.
			\end{equation}
			The derived projection formula reads
			\[
				\mathbf{R}\varphi_* \mathcal G \Lotimes \mathcal O_t \simeq \mathbf{R}\varphi_* \bigl( \mathcal G \Lotimes \mathcal O_{\mathcal X'_t} \bigr).
			\]
			Note that $\varphi$ is flat, so we do not need to derived $\varphi^* \mathcal O_t = \mathcal O_{\mathcal X'_t}$. The object $\mathcal G \Lotimes \mathcal O_{\mathcal X_t'}$ is isomorphic to the two term complex $\mathcal G \otimes \mathcal O_\mathcal Z(- \mathcal X'_t) \to \mathcal G$. Let $z \in \mathcal O_{\Delta'}$ be a function that vanishes only at $t$ with order $1$. Then, the map $\mathcal G \otimes \mathcal O_\mathcal Z(-\mathcal X'_t) \to \mathcal G$ is given by the multiplication by $s := \varphi^*(z)$, hence its kernel consists of the sections whose supports are contained in $(s=0) = \mathcal X'_t$. Since $\mathcal G$ is torsion-free, there are no such sections. This proves that $\mathcal G \Lotimes \mathcal O_t \simeq \mathcal G_t$. Thus,
			\begin{equation}\label{eq: aux: isomorphic derived functors}
				\mathbf{R}\varphi_* \mathcal G \Lotimes \mathcal O_t \simeq \mathbf{R}\varphi_* \mathcal G_t.
			\end{equation}
			\item For a complex $(E^\bullet,\, d^p \colon E^p \to E^{p+1})$, let us denote the $p$th cohomology sheaf by $\mathcal H^p(\mathcal E^\bullet) = \ker d^p / \image d^{p-1}$. The $\mathcal H^0$ in the right hand side of (\ref{eq: aux: isomorphic derived functors}) is $H^0(\mathcal G_t) \otimes \mathcal O_t$, hence it suffices to prove $\mathcal H^0(\mathbf{R}\varphi_* \mathcal G \Lotimes \mathcal O_t)=0$. We use the spectral sequence for the Tor sheaves $\varTor_{-p}(\mathcal A^\bullet, \mathcal B^\bullet) = \mathcal H^p(\mathcal A^\bullet \Lotimes \mathcal B^\bullet)$:
			\[
				E_2^{p,q} := \varTor_{-p}(\mathbf{R}^q \varphi_* \mathcal G,\, \mathcal O_t) \Rightarrow H^{p+q} := \varTor_{-p-q}(\mathbf{R}\varphi_* \mathcal G,\, \mathcal O_t). \tag{\cite[\textsection 3.3]{Huybrechts:FourierMukai}}
			\]
			Since $\mathcal O_t$ admits a locally free resolution $\mathcal O_{\Delta'}(-t) \to \mathcal O_{\Delta'}$, $E_2^{p,q}=0$ unless $p\in \{0,-1\}$. Consequently,
			\[
				\mathrm{F}^p H^n / \mathrm{F}^{p+1} H^n \simeq E_\infty^{p,n-p} = E_2^{p,n-p}.
			\]
			The filtration structure induces the short exact sequence
			\[
				0 \to E_2^{-1,1} \to H^0 \to E_2^{0,0} \to 0,
			\]
			which can be regarded as an exact sequence of vector spaces over $\C(\{t\}) \simeq \C$. Hence, $H^0 \simeq E_2^{0,0} \oplus E_2^{-1,1}$. We see that $E_2^{0,0} = \mathcal H^0( \varphi_* \mathcal G \Lotimes \mathcal O_t) = \varphi_* \mathcal G \otimes \mathcal O_t$, and $E_2^{-1,1} = \varTor_1( \mathbf{R}^1\varphi_* \mathcal G,\mathcal O_t)$. The last one can be identified with $\mathcal T \otimes \mathcal O_t$, where $\mathcal T$ is the torsion subsheaf of $ \mathbf{R}^1\varphi_* \mathcal G $. Consequently, we get an isomoprhism
			\[
				H^0 \simeq E_2^{0,0} \oplus E_2^{-1,1} \simeq (\varphi_* \mathcal G \oplus \mathcal T) \otimes \mathcal O_t.
			\]
			The arguments in {\it Step 2} are also valid for $t=0$, so (\ref{eq: aux: isomorphic derived functors}) holds for $t=0$. Looking at $\mathcal H^0$ of both sides, we get $\varphi_*\mathcal G \otimes \mathcal O_{0} \simeq H^0(\mathcal G_0) \otimes \mathcal O_0$, and it vanishes by {\it Step 1}. By semicontinuity of (\ref{eq: local rank function}), $\varphi_* \mathcal G \otimes \mathcal O_t=0$ for general $t$. Since $\mathcal T$ is torsion sheaf, $\mathcal T \otimes \mathcal O_t=0$ for general $t$. It follows that $H^0 = \mathcal H^0( \mathbf{R}\varphi_* \mathcal G \Lotimes \mathcal O_t)=0$ for general $t$. \qedhere
		\end{enumerate}
	\end{proof}
	\subsection{Example}
		Based on the arguments in \cite[\textsection 2.5.1]{Hacking:CompactModuliBarcelonaNote}, we demonstrate the relation between three block collections and $\Q$-Gorenstein degenerations of del Pezzo surfaces. For simplicity, we present a particular toric surface as an explicit example, however, the argument can be generalized to every other toric surfaces listed in \cite[Theorem~4.1]{HackingProkorov:DegenerationOfDelPezzo}.
		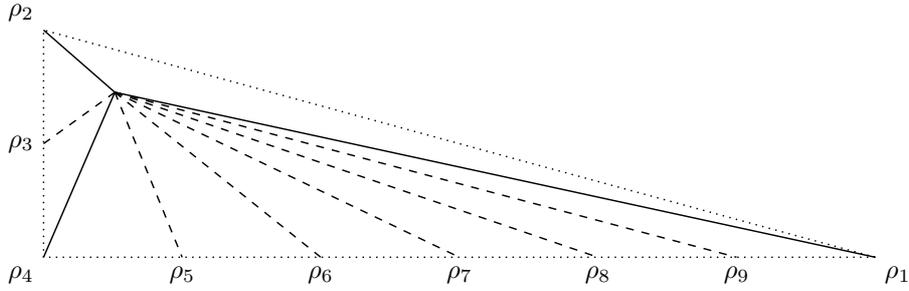
\begin{figure}[h]
		    \centering
			\begin{tikzpicture}
                \pgfmathsetmacro{\unit}{1.25}
				\coordinate (E1) at (8*\unit,-1.75*\unit);
				\coordinate (E2) at (-.75*\unit,0.66*\unit);
				\coordinate (E4) at (-.75*\unit,-1.75*\unit);
				\draw (0,0) -- (E1);
				\draw (0,0) -- (E2);
				\draw (0,0) -- (E4);
				\draw[dotted] (E1) -- (E2) -- (E4) -- cycle;
				
				\coordinate (E3) at ($(E2)!0.5!(E4)$);
				
				\draw[dashed] (0,0) -- (E3);
				\foreach \n in {5,6,7,8,9}{
				    \pgfmathtruncatemacro{\m}{\n-4}
				    \pgfmathsetmacro{\int}{1/6*\m}
				    \coordinate (E\n) at ($(E4)!\int!(E1)$);
				    \draw[dashed] (0,0) -- (E\n);
				}
				\draw node[anchor=north west] at (E1) {$\rho_1$};
				\draw node[anchor=south east] at (E2) {$\rho_2$};
				\draw node[anchor=north east] at (E4) {$\rho_4$};
				\foreach \n in {3,5,6,7,8,9}{
					\ifthenelse{\n = 3}{
						\draw node[anchor=east] at (E\n) {$\rho_\n$};
					}{
						\draw node[anchor=north] at (E\n) {$\rho_\n$};						
					}
				}
			\end{tikzpicture}
			\caption{Toric fan of $X$ (solid lines) and $\mathcal Z_0$ (solid and dashed lines)}\label{fig: eg: Toric fan}
		\end{figure}
		Let $X$ be the complete toric surface whose fan is spanned by three rays, $\rho_1 = (49,-9)$, $\rho_2 = (-5,1)$, and $\rho_4 = (-5,-9)$. These are the solid line in Figure~\ref{fig: eg: Toric fan}. According to \cite{HackingProkorov:DegenerationOfDelPezzo}, $X$ is the fake weighted projective plane corresponding to the solution $(r_1,r_2,r_3)=(2,5,9)$ of the Markov type equation
		\[
			r_1^2 + 2r_2^2 + 6r_3^2 = 6 r_1r_2r_3.
		\]
		A one-parameter $\Q$-Gorenstein smoothing $\mathcal X / (0 \in \Delta)$ has the general fiber $S := \mathcal X_t$ a del Pezzo surface of degree $3$. Since $H^2(\mathcal T_X)=0$\,(\cite[Proposition~3.1]{HackingProkorov:DegenerationOfDelPezzo}), there exists a base change $\mathcal X' / (0 \in \Delta')$ such that  Corollary~\ref{cor: simul M-res} can be applied simultaneously to all the singular points of $X$. So, there is a birational morphism $\mathcal Z / (0 \in \Delta') \to \mathcal X' / (0 \in \Delta')$ such that $\mathcal Z_0$ is the fan in Figure~\ref{fig: eg: Toric fan}, and $\mathcal Z_t \simeq \mathcal X'_t$ for general $t$. Let $P_i$ be the torus fixed point corresponding to the cone $\R_{\geq0} \cdot \rho_i + \R_{\geq0} \cdot \rho_{i+1}$ (regarding $\rho_{10}$ as $\rho_1$). Then, $(P_1 \in \mathcal Z_0) \simeq \frac{1}{2^2}(1,1)$, $(P_2,P_3 \in \mathcal Z_0) \simeq \frac{1}{5^2}(1,4)$, and $(P_4,\ldots,P_9 \in \mathcal Z_0) \simeq \frac{1}{9^2}(1,17)$. Since $\gcd(2,9)=1$, we may choose an integer $c$ such that $c \cdot \Div \rho_1$ satisfies the conditions (1--3) of Theorem~\ref{thm: Main thm}. For instance, $D_1 := 9^2 \cdot \Div \rho_1$ satisfies the conditions (1) with respect to $P_1$, and (3) with respect to other $P_k$'s. This produces an exceptional vector bundle $E_{1,t}$ of rank $2$ on $S$. For $k=2,\ldots,9$, let
		\[
			D_k = D_{k-1} + \Div \rho_k.
		\]
		It can be shown that $D_i$ satisfies the condition (1) with respect to $P_k$, (2) with respect to $P_1,\ldots,P_{k-1}$, and (3) with respect to the remaining $P_k$'s. This yields an exceptional vector bundle $E_{k,t}$ over $S$. Let $\mathcal O^r(\alpha)$ be an exceptional vector bundle on $S$ of rank $r$ and $c_1 = \alpha$ (indeed, exceptional vector bundles on $S$ are determined by $r$ and $c_1$\,\cite[Corollary~2.5]{Gorodentsev:MovingAnticacnonical}). After a suitable choice of $L \in \Pic S$, one finds that these bundles form a three block collection
		\begin{align*}
			\D(S) &= \langle\, \mathcal E_{1,t} \otimes \mathcal O(L),\, \ldots,\, \mathcal E_{9,t} \otimes \mathcal O(L) \, \rangle \\
			&= \langle\, \mathcal O^2(2K+\ell),\quad \mathcal O^5(2\ell) ,\, \mathcal O^5(K+3\ell),\quad \mathcal O^9(K+5\ell-e_1),\,\ldots,\, \mathcal O^9(K+5\ell-e_6)\, \rangle
		\end{align*}
		corresponding to the solution $(r_1,r_2,r_3)=(2,5,9)$ to the Markov type equation. Here, $e_1,\ldots,e_6 \subset S$ is a disjoint set of $(-1)$-curves and $\ell \subset S$ is the pullback of a general line along the contraction $S \to \P^2$ of $e_1,\ldots,e_6$.
	\subsection*{Acknowledgments}{\ } The author is grateful to Mohammad Akhtar, Seung-Jo Jung, Kyoung-Seog Lee, Yongnam Lee, Dongsoo Shin, and Giancarlo Urz\'ua for valuable discussions from which he took a lot of benefit. A part of this work had been done when the author was a member of Universit\"at Bayreuth -- Lehrstuhl Mathematik VIII.
	
	This work was supported by ERC Advanced Grant no. 340258 TADMICAMT, and by KIAS Individual Grant no. MG074601 at Korea Institute for Advanced Study.
\bibliography{arXiv_OrthogonalCollection}
\end{document}